\documentclass[12pt]{article}
\input epsf.tex

\usepackage{graphicx}
\usepackage{amsmath,amsthm,amsfonts,amscd,amssymb,comment,eucal,latexsym,mathrsfs}
\usepackage{stmaryrd}
\usepackage[all]{xy}

\usepackage{epsfig}

\usepackage[all]{xy}
\xyoption{poly}
\usepackage{fancyhdr}
\usepackage{wrapfig}
\usepackage{epsfig}

\usepackage[pdftex]{hyperref}

\theoremstyle{plain}
\newtheorem{thm}{Theorem}[section]
\newtheorem{prop}[thm]{Proposition}
\newtheorem{lem}[thm]{Lemma}

\theoremstyle{definition}
\newtheorem{defn}{Definition}
\theoremstyle{remark}
\newtheorem{remark}{Remark}

\newtheorem{notation}{Notation}

\advance \topmargin by -\headheight
\advance \topmargin by -\headsep
\evensidemargin \oddsidemargin
  \def\C{{\mathbb{C}}}           \def\N{{\mathbb{N}}}    \def\R{{\mathbb{R}}}        

\def\bfa{{\bf{a}}}                         

 \def\cB{{\mathcal{B}}} \def\cC{{\mathcal{C}}}            \def\cO{{\mathcal{O}}} \def\cP{{\mathcal{P}}}   \def\cS{{\mathcal{S}}}      \def\cY{{\mathcal{Y}}} 

            \def\tM{{\widetilde{M}}}        \def\tU{{\widetilde{U}}}     

           \def\tSigma{{\widetilde{\Sigma}}}  

 \def\tdelta{{\widetilde{\delta}}}               

\newcommand{\G}{\Gamma}

\newcommand{\Si}{\Sigma}
\newcommand{\tSi}{\widetilde{\Sigma}}
\newcommand{\eps}{\epsilon}

\renewcommand\a{\alpha}
\renewcommand\b{\beta}
\renewcommand\d{\delta}
\newcommand\g{\gamma}
\renewcommand\k{\kappa}

\newcommand\avg{\operatorname{avg}}

\newcommand\bvol{\overline{\operatorname{vol}}}

\newcommand\cov{\operatorname{Cov}}

\newcommand\dom{\operatorname{dom}}

\newcommand\Haar{\mathrm{Haar}}

\newcommand\Map{{\operatorname{Map}}}

\newcommand\ovol{\operatorname{vol}}

\newcommand\Prob{\operatorname{Prob}}

\newcommand\rng{\operatorname{rng}}

\newcommand\Stab{\operatorname{Stab}}
\newcommand\sep{\operatorname{Sep}}
\newcommand\Span{\operatorname{Span}}

\newcommand\vol{\mathrm{vol}}

\def\cc{{\curvearrowright}}

\newcommand{\resto}{\upharpoonright}

  \newcommand{\dee}{\mathrm{d}}

\begin{document}
\title{Locally compact sofic entropy theory}
\author{Lewis Bowen\footnote{supported in part by NSF grant DMS-2154680}\\ University of Texas at Austin}
\maketitle

\begin{abstract}
This paper generalizes sofic entropy theory, in both the topological and measure-theory settings, to actions of locally compact groups. We prove invariance under topological and measure conjugacy of these entropies and establish the variational principle.
\end{abstract}


\noindent
\tableofcontents

\section{Introduction}

The entropy of a measure-preserving transformation $T:X \to X$ on a probability space $(X,\cB,\mu)$ with respect to a measurable map $\phi:X \to A$ (where $A$ is finite or countable) is defined as follows. First we consider the map $\phi^n:X \to A^n$
$$\phi^n(x) = (\phi(x), \phi(Tx),\ldots, \phi(T^{n-1}x)).$$ 
Then pushforward the measure $\mu$ to obtain a measure $\phi^n_*\mu$ on $A^n$. The Shannon entropy of this measure is defined by
$$H(\phi^n_*\mu) = \sum_{\bfa \in A^n} - \phi^n_*\mu(\bfa) \log (\phi^n_*\mu(\bfa) ).$$
Finally, the entropy rate of $T$ with respect to $\mu$ and $\phi$ is 
$$h_\mu(T,\phi) = \lim_{n\to\infty} n^{-1}H(\phi^n_*\mu).$$
For example, if $T$ represents time and $\phi$ represents a measurement on the system then $H(\phi^n_*\mu)$ is the expected amount of information gained from the measurements up to time $n$. So the entropy rate is the expected information gain per unit time.

The entropy of $T$, denoted $h_\mu(T)$, is the supremum of $h_\mu(T,\phi)$ over all observables $\phi$ with finite range. It is a measure-conjugacy invariant and is crucial for classification purposes. In fact, Kolmogorov introduced entropy in order to prove the existence of Bernoulli shifts which are not measurably conjugate (because they have different entropies) \cite{kolmogorov-1958, kolmogorov-1959}. Sinai proved that, if $T$ is ergodic, then $T$ factors onto a Bernoulli shift with the same entropy \cite{sinai-weak} and Ornstein proved that two Bernoulli shifts are measurably conjugate if and only if they have the same entropy \cite{ornstein-1970a}.

Since these classical results were obtained, entropy theory has been generalized to actions of countable amenable groups \cite{kieffer-1975a, OW80}, $\R^n$ \cite{MR597458}, locally compact amenable groups \cite{OW87, avni-2010} and countable sofic groups \cite{bowen-jams-2010, kerr-li-variational, MR3616077, MR4138907}. The goal of this paper is to generalize entropy theory to actions of locally compact sofic groups. 

\subsection{Sofic groups and the residually finite case}
A countable group $\G$ is called {\bf sofic} if it admits a sofic approximation, which is a sequence of partial actions on finite sets which locally approximates the action of $\G$ on itself by left-translations \cite{MR3616077, pestov-sofic-survey, capraro-lupini}. To illustrate, suppose $\G$ is residually finite. This means there exist normal finite-index subgroups $N_1 \ge N_2 \ge \cdots$ in $\G$ such that $\cap_i N_i = \{1_\G\}$. The sequence $\Si$ of actions $\G \cc \G/N_i$  is a special kind of {\bf sofic approximation} to $\G$ called a {\bf transitive sofic approximation} or a {\bf sofic approximation by homomorphisms}.  To motivate our approach, we first review sofic entropy of $\G$-actions with respect to such a sequence. 

Let $(X,\rho_X)$ be a compact metric space on which $\G$ acts by homeomorphisms. An orbit $\G x \subset X$ is {\bf periodic} if the stabilizer of $x$, $\Stab_\G(x)=\{g \in \G:~gx = x\}$ has finite index in $\G$. We can also think of a periodic orbit as the image of a $\G$-equivariant map $\phi:\G/H \to X$ for some finite index subgroup $H$. 

While the periodic orbits are certainly of interest, they are too restrictive for entropy purposes. Instead, we consider {\bf almost periodic orbits}. These are quantified by a finite subset $U \subset \G$ and an error tolerance $\d>0$. A map $\phi:\G/N_i \to X$ is {\bf $(\rho_X,U,\d)$-equivariant} if 
$$|\G/N_i|^{-1} \sum_{hN_i \in \G/N_i} \rho_X(\phi(g hN_i), g \phi(hN_i)) < \d$$
for all $g \in U$. Let $\Map(\G/N_i, X, \rho_X: U,\d)$ be the set of all $(\rho_X,U,\d)$-equivariant maps of $\G/N_i$ into $X$. 

In order to quantify the `size' of $\Map(\G/N_i, X, \rho_X: U,\d)$, we introduce a metric on it. In general, if $\phi, \psi$ are maps from $\G/N_i$ to $X$ then we set
$$\rho_X^{\G/N_i}(\phi,\psi) = |\G/N_i|^{-1} \sum_{hN_i \in \G/N_i} \rho_X(\phi(hN_i),  \psi(hN_i)).$$
Let $\sep_\eps( \Map(\G/N_i, X, \rho_X: U,\d))$ denote the maximum cardinality of a $(\rho_X^{\G/N_i},\eps)$-separated subset of $ \Map(\G/N_i, X, \rho_X: U,\d)$. 

Finally, we can define topological $\Si$-entropy by
$$h_\Si(\G,X) =\lim_{\eps \searrow 0} \inf_{U\subset \G} \inf_{\d>0} \limsup_{i\to\infty} |\G/N_i|^{-1} \log \sep_\eps(\Map(\G/N_i, X, \rho_X : U,\d))$$
where the first infimum is over all finite subsets $U \subset \G$. 

To define measure entropy, let $\Prob(X)$ denote the space of Borel probability measures on $X$ endowed with the weak* topology. This is the smallest topology such that for every continuous $f:X \to \R$, the map $\mu \mapsto \int f~\dee\mu$ is continuous on $\Prob(X)$. Given an open set $\cO \subset \Prob(X)$, let $\Map(\G/N_i, X, \rho_X: U,\d,\cO)$ be the set of $\phi \in \Map(\G/N_i, X, \rho_X: U,\d)$ such that if $u_i$ denotes the uniform probability measure on $\G/N_i$ then $\phi_*u_i \in \cO$ where $\phi_*u_i$ is called the {\bf empirical measure} of $\phi$. 

Let $\mu$ be a $\G$-invariant probability measure on $X$. The measure $\Si$-entropy of the action $(\G,X,\mu)$ is defined by
$$h_\Si(\G,X,\mu) =\lim_{\eps \searrow 0} \inf_{\cO \ni \mu} \inf_{U\subset \G} \inf_{\d>0} \limsup_{i\to\infty} |\G/N_i|^{-1} \log \sep_\eps(\Map(\G/N_i, X, \rho : U,\d, \cO))$$
where $\cO$ varies over all open neighborhoods of $\mu$ in $\Prob(X)$ and, as above, $U$ varies over finite subsets of $\G$.

The approach to entropy theory through $\eps$-separating numbers of partial orbits was first introduced by Rufus Bowen \cite{MR0274707}. It was introduced into the sofic context by Kerr and Li \cite{kerr-li-variational}. 

Some of the main results about these entropies are: topological $\Si$-entropy is a topological conjugacy invariant while measure $\Si$-entropy is a measure conjugacy invariant. The variational principle holds: topological $\Si$-entropy is the supremum of $h_\Si(\G,X,\mu)$ over all $\G$-invariant measures $\mu$. Moreover the entropies of Bernoulli shifts, Gaussian actions and many algebraic actions have been computed \cite{bowen-jams-2010, MR2813530, MR3693125, hayes-fk-determinants}. These entropies agree with their classical counterparts when $\G$ is amenable \cite{kerr-li-sofic-amenable}. See \cite{MR3616077, MR4138907} for more results.

\subsection{Locally compact sofic groups}

A locally compact group $G$ is {\bf sofic} if it admits a {\bf sofic approximation} \cite{MR4555895}. Roughly speaking, a sofic approximation is a sequence of partial actions of $G$ on topological spaces which locally approximates the action of $G$ on itself by left-translations. This is reviewed in \S \ref{S:lcsg} below. For example, suppose $G$ admits a lattice\footnote{A subgroup $\G$ is a {\bf lattice} if it is discrete in $G$ and there is a subset $F \subset G$ with finite left-Haar measure such that $G = \G F$.} subgroup $\G$ which is residually finite. Let $\{\G \cc \G/N_i\}_{i=1}^\infty$ be a sofic approximation to $\G$ as in the previous subsection. Then the sequence $\tSigma=\{G \cc G/N_i\}_{i=1}^\infty$ of homogeneous actions is a sofic approximation to $G$.

For the sake of motivation, we describe sofic entropy with respect to $\tSigma$; the general case is explained in the paper. Actually, the definitions are straightforward generalizations of the countable case. Suppose $G \cc (X,\rho_X)$ is an action on a compact metric space. Given a pre-compact open set $U \subset G$ and an error tolerance $\d>0$, a measurable map $\phi:G/N_i \to X$ is said to be {\bf $(\rho_X,U,\d)$-equivariant} if 
$$\int_{G/N_i} \rho_X(\phi (gp), g\phi(p))~\dee\bvol(p) < \d$$
where $\bvol$ is the unique $G$-invariant Borel probability measure on $G/N_i$. If $\psi: G/N_i \to X$ is another measurable map then define its distance to $\phi$ by
$$\rho^{G/N_i}_X(\phi,\psi) = \int_{G/N_i} \rho_X(\phi (p), \psi(p))~\dee\bvol(p).$$
Finally, define the topological $\tSi$-entropy by
$$h_{\tSi}(G,X) =\lim_{\eps \searrow 0} \inf_{U\subset G} \inf_{\d>0} \limsup_{i\to\infty} \vol(G/N_i)^{-1} \log \sep_\eps(\Map(G/N_i, X, \rho_X : U,\d))$$
where 
\begin{enumerate}
\item $\Map(G/N_i, X, \rho_X : U,\d)$ is the set of $(\rho_X,U,\d)$-equivariant maps;
\item $\sep_\eps( \Map(G/N_i, X, \rho_X : U,\d))$ is the maximum cardinality of a $(\rho_X^{G/N_i}, \eps)$-separated subset of  $\Map(G/N_i, X, \rho_X : U,\d)$;
\item  the first infimum is over all pre-compact open subsets $U \subset G$;
\item $\vol(G/N_i)$ is the volume of $G/N_i$ with respect to a fixed left-invariant Haar measure on $G$.
\end{enumerate}

Similarly, if $\cO \subset \Prob(X)$ is an open set, then we let $\Map(G/N_i, X, \rho_X : U,\d,\cO)$ be the set of all $\phi \in \Map(G/N_i, X, \rho_X : U,\d)$ such that $\phi_*\bvol \in \cO$ where $\phi_*\bvol(E) = \bvol(\phi^{-1}(E))$ is the pushforward measure. 

Let $\mu$ be a $G$-invariant probability measure on $X$. The measure $\tSi$-entropy of the action $(G,X,\mu)$ is defined by
$$h_{\tSi}(G,X,\mu) =\lim_{\eps \searrow 0} \inf_\cO \inf_{U\subset G} \inf_{\d>0} \limsup_{i\to\infty} \vol(G/N_i)^{-1} \log \sep_\eps(\Map(G/N_i, X, \rho : U,\d, \cO))$$
where $\cO$ varies over all open neighborhoods of $\mu$ in $\Prob(X)$.

Our main results are: topological $\tSi$-entropy is invariant under topological conjugacy, measure $\tSi$-entropy is invariant under measure conjugacy and the variational principle holds. 
 
In the paper, we work in a more general setting. First, we do not require that $X$ is compact. The reason for this is that there are many actions of locally compact groups,  such as the action on the space of point measures, where the action space $X$ is not compact but the action map $G \times X \to X$ given by $(g,x) \mapsto gx$ is uniformly continuous. Moreover, the arguments do not require compactness as long as $X$ admits a Polish metric with respect to which the action is uniformly continuous. 

Second, we do not require that $\rho_X$ is a metric on $X$. Instead, we assume $\rho_X$ is a bounded, uniformly continuous, pseudo-metric which is {\bf dynamically generating}. The latter means that if $x, y \in X$ and $x \ne y$ then there exists $g \in G$ with $\rho_X(gx,gy)>0$. 
 
\subsection{About the proofs}
The rough outlines of the proofs are similar to the countable case (as in  \cite{MR3616077} or as outlined in \cite{ MR4138907}), however there are some significant differences highlighted here in the special case of homogeneous sofic approximations.

In the case of topological entropy, we say that a measurable map $\phi:G/N_i \to X$ is {\bf $(\rho_X,U,\d)$-equivariant on average} if 
$$\int_U \int_{G/N_i} \rho_X(\phi (gp), g\phi(p))~\dee\bvol(p) \dee \Haar(g)< \d \Haar(U)$$
where $\Haar$ is a fixed left-Haar measure on $G$. Let $\Map_{\avg}(G/N_i, X, \rho_X : U,\d)$ be the set of such maps. We prove that topological $\tSi$-entropy does not change if we replace $\Map(\cdot)$ with $\Map_{\avg}(\cdot)$. The benefit is that it is easier to construct maps which satisfy this weaker equivariance and this helps establish lower bounds on topological $\tSi$-entropy.

In the case of measure entropy, we say a map $\phi:G/N_i \to X$ is {\bf measure-preserving} if $\phi_*\bvol = \mu$. Let $\Map_{\textrm{mp}}(G/N_i, X, \rho_X : U,\d) \subset \Map(G/N_i, X, \rho_X : U,\d)$ be the subset of such maps. We prove that measure $\tSi$-entropy does not change if we replace $\Map(\cdot)$ with $\Map_{\textrm{mp}}(\cdot)$, when the sofic approximation has no atoms. This helps establish upper bounds on measure $\tSi$-entropy.

Our proof of the Variational Principle is streamlined. It appears to be shorter than previous proofs even in the case when the acting group is countable.

\subsection{Organization}

\begin{itemize}
\item \S \ref{S:lcsg} reviews sofic approximations to locally compact groups;
\item \S \ref{S:pttet} goes over preliminary concepts such as the $\eps$-separated number, and induced pseudo-metrics on spaces of maps needed in \S \ref{S:tse};
\item \S \ref{S:tse} proves topological $\Si$-entropy is invariant under topological conjugacy;
\item \S \ref{S:ptmet} reviews the weak* topology and measure-conjugacy for use in \S \ref{S:mse};
\item \S \ref{S:mse} proves measure $\Si$-entropy is invariant under measure conjugacy;
\item \S \ref{S:vp} establishes the Variational Principle;
\item \S \ref{S:oq} is a list of open problems;
\item \S \ref{S:ni} is a notation index.
 \end{itemize}

Future parts of this work are planned which will contain a variety of tools for estimating sofic entropy and specific examples of entropy computations.




{\bf Acknowledgements}. It is a pleasure to thank Benjy Weiss for several enlightening email exchanges and to Jean-Paul Thouvenot for an inspiring question (which, hopefully, will be answered in future work). Locally compact sofic entropy theory formed the bulk of Sukhpreet Singh's unpublished PhD thesis (under my direction) and I am grateful for the many hours we spent together working on the theory. The definitions in this paper are different from those in his thesis because, with Peter Burton, we re-worked the concept of a sofic approximation to a locally compact group. However, the resulting notion of entropy is equivalent.

\section{Locally compact sofic groups}\label{S:lcsg}
This section recalls definitions and notation from \cite{MR4555895} regarding locally compact sofic groups. Proofs are in \cite{MR4555895}. 

\subsection{Local $G$-spaces} \label{subsec.localgspace}
We use the abbreviation lcsc to mean locally compact second countable. Let $G$ be an lcsc group.

\begin{defn}\label{D:partial}
A {\bf partial left-action} of $G$ on a Hausdorff space $M$ is a continuous map $\a:\dom(\a) \to M$ where $\dom(\a) \subset G\times M$ is open.  We require the following axioms hold for all $p\in M$.
\begin{enumerate}
\item[Axiom 1.] $(1_G,p) \in \dom(\a)$ and $\a(1_G,p) = p$. \label{D:partial-identity}
\item[Axiom 2.] If $(g,p)\in \dom(\a)$ then $(g^{-1}, \a(g,p)) \in \dom(\a)$ and $\a(g^{-1}, \a(g,p))  = p$.  \label{D:partial-inverse}
\item[Axiom 3.] If $(h,p), (g,\a(h,p)), (gh,p) \in \dom(\a)$ then $\a(gh,p)=\a(g,\a(h,p))$. \label{D:partial-mult}
\end{enumerate}
A partial action $\a$ is {\bf homogeneous} if in addition it satisfies the following.
\begin{enumerate}
\item[Axiom 4.] \label{D:partial-homeo} For every $p\in M$ there is an open neighborhood $O_p$ of  $1_G$ in $G$ such that $O_p \times \{p\}  \subset \dom(\a)$ and the  
restriction of $\a(\cdot, p)$ to $O_p \times \{p\}$ is a homeomorphism onto an open neighborhood of $p$ in $M$. 
\end{enumerate}

\end{defn}

\begin{remark}
The paper \cite{MR4555895} works with partial right-actions instead of left-actions. It is straightforward to switch from right-actions to left-actions. Since it is more convenient to work with left-actions in this paper, we do so. 
\end{remark}

\begin{defn}
A {\bf local left-$G$-space} is a pair $(M,\a)$ where $M$ is an lcsc space and $\a$ is a partial homogeneous left-action. 
\end{defn}

\begin{notation}
Because all of the actions in this paper are left actions, we will simply call a pair $(M,\a)$ as above a {\bf local $G$-space}. We will usually such a space by $M$, leaving the action $\a$ implicit. To simplify notation, we write $g.p= \a(g,p)$. If $K \subset M$, we write $g.K = \{ \a(g,k):~k \in K\}$. In particular, $g.K$ is well-defined if and only if $\{g\}\times K$ is in the domain of the action $\a$. Similarly, we write $O.p = \{\a(g,p):~g\in O\}$ if $O\times \{p\} \subset \dom(\a)$. 
\end{notation}

\begin{remark}
We write $g.h.p = g.(h.p)$. It is not necessarily true that $g_1.g_2.g_3.p = g_1g_2g_3.p$ even when both sides are well-defined. For counterexamples, see  \cite{MR4555895}.
\end{remark}

\begin{defn}
Let $(M,\a)$ be a local $G$-space and $p \in M$. A {\bf chart centered at $p$} is a homeomorphism $f_p:\dom(f_p) \to \rng(f_p)$ where $\dom(f_p) \subset M$ is an open neighborhood of $p$, $\rng(f_p)$ is an open neighborhood of the identity in $G$ and $g=f_p(g.p)$ for all $g \in \rng(f_p)$. In particular, we require that $g.p$ is well-defined for all $g\in \rng(f_p)$. By Axiom 4 of Definition \ref{D:partial}, for every $p \in M$ there exists a chart centered at $p$.  
\end{defn}

The next results are in  \cite{MR4555895}.
\begin{prop}[The canonical measure]\label{P:canonicalmeasure}
Let $(M,\a)$ be a local left-$G$-space. Fix a right-Haar measure $\Haar$ on $G$. Then there exists a unique Radon measure $\vol_M$ on $M$ satisfying the following. If $p \in M$, $f_p$ is a chart centered at $p$ and $K \subset \dom(f_p)$ is Borel then 
\begin{eqnarray}\label{E:vol}
 \vol_M(K) = \Haar(\{g \in \rng(f_p):~ g.p  \in K\}) = \Haar(f_p(K)).
 \end{eqnarray}
\end{prop}
We write $\vol$ instead of $\vol_M$, when the choice of $M$ is clear.

\begin{lem}[Locally measure-preserving]\label{L:mp}
Let $(M,\a)$ be a local $G$-space and suppose $\{g\} \times  K \subset \dom(\a)$ for some measurable $K \subset M$ and $g \in G$. If $G$ is unimodular then 
$$\vol_M(g.K) = \vol_M(K).$$
\end{lem}

\begin{defn}\label{D:sofic group}
Let $M=(M,\a)$ be a local $G$-space and let $U \subset G$ be open and pre-compact and let $\eps>0$.  Let $M[U]=M[\a,U]$ be the set of all $p \in M$ such that if $g,h \in G$ are such that $g, h, gh \in U$ then $g.h.p=gh.p$ (in particular, both sides are well-defined). Moreover, we require that the map $g \mapsto \a(g,p)$ defines a homeomorphism from $U$ to an open neighborhood of $p$. We say $M$ is a {\bf $(U,\eps)$-sofic approximation to $G$} if $\vol_M(M)<\infty$ and
$$\vol_M(M[U]) \ge (1-\eps) \vol_M(M).$$
\end{defn}


\begin{defn} A \textbf{sofic approximation} to $G$ is a sequence $\Sigma = (M_i)_{i=1}^\infty$ where $M_i$ is a $(U_i,\epsilon_i)$-sofic approximation such that the $U_i$ are pre-compact open sets increasing to $G$ and the sequence $\epsilon_i$ decreases to $0$. We say $G$ is {\bf sofic} if it admits a sofic approximation. 
\end{defn}

By \cite{MR4555895}, all sofic groups are unimodular. It is unknown whether all unimodular lcsc groups are sofic.

\section{Preliminaries on uniform actions, pseudo-metrics, separated and spanning sets, and model spaces}\label{S:pttet}


\subsection{Uniform actions and conjugacies}
If $(Z,d_Z)$ is a metric space then an action $G \cc Z$ by homeomorphisms is {\bf uniform} if the map $G\times Z \to Z$ is uniformly continuous. This means: for every $\eps>0$ there exists $\d>0$ such that 
$$d_G(g_1,g_2) + d_Z(z_1,z_2) < \d \Rightarrow d_Z(g_1z_1,g_2z_2) < \eps$$
for all $g_1,g_2 \in G$ and $z_1,z_2\in Z$ where $d_G$ is a fixed left-invariant proper continuous metric on $G$. 

Suppose we are given two uniform actions $G \cc Z$ and $G \cc Y$. A map $\Phi:Y \to Z$ is a {\bf uniform conjugacy} if it is uniformly continuous, invertible, $\Phi^{-1}:Y \to Z$ is uniformly continuous and $\Phi(gz)=g\Phi(z)$ for all $g\in G$ and $z\in Z$. 


\subsection{Pseudo-metrics and separating/spanning/covering sets}
A {\bf pseudo-metric} $\rho:X\times X \to [0,\infty)$ on a set $X$ is a metric on $X$ except that we allow $\rho(x,y)=0$ even if $x \ne y$.  A pseudo-metric $\rho$ is 
\begin{itemize}
\item {\bf 1-bounded} if $\rho(x,y)\le 1$ for all $x,y \in X$;
\item {\bf generating} for an action $G \cc X$ if for every $x \ne y$ there exists $g \in G$ such that $\rho(gx,gy)>0$. 
\end{itemize}
If $(X,d_X)$ is a metric space then a pseudo-metric $\rho$ on $X$ is {\bf uniformly continuous} if for every $\eps>0$ there exists $\d>0$ such that $\rho(x,y)<\d$ implies $d_X(x,y)<\eps$.

If $(X,\rho)$ is a pseudo-metric space then a subset $Y \subset X$ is 
\begin{itemize}
\item {\bf $(\rho,\eps)$-separated} if $\rho(x,y)> \eps$ for all $x, y \in Y$ with $x \ne y$;
\item {\bf $(\rho,\eps)$-spanning} if for all $x\in X$ there exists $y \in Y$ with $\rho(x,y)< \eps$.
\end{itemize}
We say that a collection $\cC$ of subsets of $X$ is a {\bf $(\rho,\eps)$-cover} if the union of sets in $\cC$ is $X$ and each $C \in \cC$ has diameter $<\eps$.

Let 
\begin{itemize}
\item $\sep_\eps(X,\rho)$ be the maximum cardinality of a $(\rho,\eps)$-separated subset;
\item $\Span_\eps(X,\rho)$ be the minimum cardinality of a $(\rho,\eps)$-spanning subset;
\item $\cov_\eps(X,\rho)$ denote the minimum cardinality of a $(\rho,\eps)$-cover. 
\end{itemize}

It is well-known (and an easy exercise) that
\begin{eqnarray}\label{E:covsepspan}
\cov(\rho,2\eps) \le \Span(\rho,\eps) \le \sep(\rho,\eps)  \le \cov(\rho,\eps).
\end{eqnarray}

\subsection{Pseudo-metrics on model spaces}\label{S:rhoM}

\begin{notation}
Given a finite measure space $(M,\vol)$ we let $\bvol$ denote the probability measure $\bvol(E) = \frac{\vol(E)}{\vol(M)}$. 
\end{notation}

Suppose $(M,\ovol)$ is a finite measure space and $X$ is a Hausdorff space. Let $\Map(M,X)$ denote the space of all measurable maps from $M$ to $X$. If $\rho$ is a pseudo-metric on $X$ then  let $\rho^M$ be the pseudo-metric on $\Map(M,X)$ defined by
$$\rho^M(\phi,\psi) =  \ovol(M)^{-1} \int \rho(\phi(p),\psi(p))~\dee\vol(p) =  \int \rho(\phi(p),\psi(p))~\dee\bvol(p) .$$

More generally, suppose $\phi, \psi$ are measurable maps to $X$, but the domains of $\phi$ and $\psi$ are not necessarily all of $M$. Let $D \subset M$ be the intersection of the domains of $\phi$ and $\psi$.  Then define 
$$\rho^M(\phi,\psi) =  \bvol(M \setminus D)+\int_{p \in D} \rho(\phi(p),\psi(p))~\dee\bvol(p).$$
Then $\rho^M$ is symmetric and satisfies the triangle inequality but $\rho^M(\phi,\phi)$ can be positive and $\rho^M(\phi,\psi)$ can be zero even if $\phi \ne \psi$. So $\rho^M$ is a quasi-metric.







\section{Topological sofic entropy}\label{S:tse}

Let $(X,d_X)$ be a Polish metric space and $G\cc X$ a uniform action. Let $\rho:X\times X \to \R$ be a continuous pseudo-metric. In this section, we define the topological $\Si$-entropy of the action and prove it is invariant under uniform conjugacy.


\begin{defn}
Let $M$ be a finite volume local $G$-space and $\phi:M \to X$. For $g \in G$ define $\phi \circ g$ by $(\phi \circ g)(p) = \phi(g.p)$ for every $p\in M$  such that $g.p$ is well-defined. 
 We also define $g\circ \phi:M \to X$ in the associative way $(g\circ \phi)(p) = g(\phi(p))$.
\end{defn}

\begin{defn}
Let $U\subset G$ be a pre-compact neighborhood of identity, $\eps >0$. Suppose that $M$ is a finite volume local $G$-space. A map $\phi:M \to X$ is {\bf $(U,\d,\rho^M)$-equivariant} if
 $$\rho^M\left(\phi \circ g, g\circ \phi\right)  < \delta$$
 for all $g\in U$. Denote by $\Map\left(M,X, \rho: U,\delta \right) \subset \Map(M,X)$ the set of all $(U,\d,\rho^M)$-equivariant maps.  
 
 
  
 \end{defn}
 
 \begin{defn}
To simplify notation, we write $\sep_\epsilon\left(\Map\left(M,X, \rho:U,\delta \right)\right)$ for the maximum cardinality of a $(\rho^{M},\eps)$-separated subset of $\Map\left(M,X, \rho:U,\delta \right)$. We define $\Span_\eps(\cdot)$ and $\cov_\eps(\cdot)$ similarly. 
 \end{defn}


\begin{defn}\label{MapEntropy1}
Let $G \cc X$ and $\rho$ be as above. Also let $\Sigma=\{M_i\}_{i=1}^\infty$ be a sofic approximation to $G$.  Define 
\begin{eqnarray*}
h^\eps_\Si\left(G, X, \rho\right)&=& \inf_{U\subset G}\inf_{\d>0}\limsup_{i\to \infty}\frac{1}{\ovol\left(M_i\right)}\log \sep_\epsilon\left(\Map\left(M_i,X, \rho:U,\delta \right)\right) \\
h_\Si\left(G, X, \rho\right)&=& \lim_{\eps \searrow 0} h^\eps_\Si\left(G, X, \rho\right)
\end{eqnarray*}
where the first infimum is over all pre-compact neighborhoods $U$ of the identity in $G$. By (\ref{E:covsepspan}), we can replace $\sep$ with $\Span$ or $\cov$ without changing the values of $h^\eps_\Si$ or $h_\Si$. 
\end{defn}

\begin{remark}
The quantity $\sep_\epsilon\left(\Map\left(M_i,X, \rho:U,\delta \right)\right)$ decreases monotonically as $\d$ decreases or $U$ increases. It increases as $\eps$ decreases. So an equivalent definition of entropy is:
$$h_\Si\left(G, X, \rho\right) = \lim_{\eps \searrow 0}\lim_{U\nearrow G}\lim_{\d\searrow 0}\limsup_{i\to \infty}\frac{1}{\ovol\left(M_i\right)}\log \sep_\epsilon\left(\Map\left(M_i,X, \rho:U,\delta \right)\right).$$

\end{remark}

Our main theorem is that topological sofic entropy is invariant under uniform conjugacy.

\begin{thm}\label{thm:top}
Let $G \cc X$, $G \cc Y$ be uniform actions on Polish metric spaces $(X,d_X), (Y,d_Y)$. Also let $\rho_X, \rho_Y$ be uniformly continuous, generating, 1-bounded pseudo-metrics on $X$ and $Y$ respectively. Suppose the actions are uniformly conjugate. Then 
$$h_\Si\left(G, X, \rho_X\right) = h_\Si\left(G, Y, \rho_Y\right)$$
for any sofic approximation $\Si$. 
\end{thm}

\begin{defn}
We will often write $h_\Si(G,X)$ for $h_\Si\left(G, X, \rho_X\right)$ where $\rho_X$ is implicit. This invariant is called the {\bf $\Si$-entropy} or {\bf topological sofic entropy} of the action $G\cc X$. It may depend on the choice of sofic approximation $\Si$, however this choice is often left implicit in informal discussion.
\end{defn}

\begin{remark}
If $G$ is a countable group then the definition of $h_\Si(G,X)$ reduces in a straightforward way to the standard definition of topological sofic entropy, as presented in \cite{MR3616077} for example. 
\end{remark}

\begin{notation}
To avoid being overly verbose, we use the following conventions. The letter $U, \tU$ and variants always denotes a pre-compact open neighborhood of the identity while $\d, \tdelta, \eps, \eps'$ etc. are always positive.
\end{notation}




\subsection{Equivariance on average}

To prove Theorem \ref{thm:top}, we first weaken the notion of a good map from ``for all $g \in U$'' to an average. For this section, we write $\rho=\rho_X$. We say a map $\phi:M \to X$ is {\bf $(U,\d,\rho^M)$-equivariant-on-average} if 
$$\int_U \rho^M(g \circ \phi, \phi \circ g) ~\dee\Haar(g) < \d \Haar(U).$$
Let $\Map_{\avg}\left(M,X, \rho:U,\delta \right)$ be the space of all such maps. Also, let
$$h_{\Si,\avg}\left(G, X, \rho\right) = \lim_{\eps \searrow 0} \inf_{U\subset G}\inf_{\d>0}\limsup_{i\to \infty}\frac{1}{\ovol\left(M_i\right)}\log \sep_\epsilon\left(\Map_{\avg}\left(M_i,X, \rho:U,\delta \right)\right).$$

\begin{prop}\label{P:weakened}
Let $\rho$ be a uniformly continuous, 1-bounded pseudo-metric on $X$. Then 
$$h_\Si\left(G, X, \rho\right) = h_{\Si,\avg}\left(G, X, \rho\right).$$
\end{prop}

To prove this we will need the following lemma. Let $B(r) \subset G$ denote the open ball of radius $r$ centered at the identity in $G$. 
\begin{lem}\label{L:group}
Let $0<r_0<r_1<r_2$ and $\d>0$. Suppose $r_0+r_1<r_2$ and
$$\d\Haar(B(r_2)) < (1-\d)\Haar(B(r_0)).$$
If $W_i \subset B(r_i)$ are any measurable subsets with 
$$\Haar(W_i) > (1-\d)\Haar(B(r_i))$$
for $i=0,2$ then $W_0W_2 \supset B(r_1)$. 
 \end{lem}

\begin{proof}
Because $G$ is sofic, it is unimodular. So it suffices to show $W_0^{-1}W_2 \supset B(r_1)$. Equivalently, we must show that for every $g \in B(r_1)$, $W_0g \cap W_2 \ne \emptyset$. Thus it suffices to show
$$\Haar(W_0g \cup W_2) < \Haar(W_0) +  \Haar(W_2).$$
This follows from:
$$\Haar(W_0g \cup W_2) \le \Haar(B(r_2)) \le (1-\d)[\Haar(B(r_0)) + \Haar(B(r_2))] <  \Haar(W_0) +  \Haar(W_2).$$
The first equality holds because $B(r_2) \supset W_0g \cap W_2$, while the other inequalities hold by hypothesis.
\end{proof}

\begin{proof}[Proof of Proposition \ref{P:weakened}]
Because $\Map_{\avg}\left(M_i,X, \rho:U,\delta \right) \supset \Map\left(M_i,X, \rho:U,\delta \right)$, it is immediate that
$$h_\Si\left(G, X, \rho\right) \le h_{\Si,\avg}\left(G, X, \rho\right).$$
To obtain the opposite inequality, it suffices to show that for every $U,\d$ there exists $\tU$, $\tdelta$, $I$ such that if $i>I$ then $\Map_{\avg}\left(M_i,X, \rho:\tU,\tdelta \right) \subset \Map\left(M_i,X, \rho:U,\delta \right)$. 

Let $U,\d$ be given. Without loss of generality, $\d<1$. We choose parameters $\k,r_0,r_1,r_2, \tdelta, \tU, i$ as follows. Because $\rho$ and the action are uniformly continuous there exist $\k>0$ and $r_0>0$ such that if $x,y \in X$ satisfy $\rho(x,y)<\k$ and $g \in B(r_0)$ then $\rho(gx,gy)<\d/10$. 

Let $r_1>0$ be such that $U \subset B(r_1)$. Choose  $r_2>r_0+r_1$. Let $\tdelta>0$ be such that 
\begin{eqnarray}\label{E:tdelta1}
\tdelta\Haar(B(r_2)) < (1-\tdelta)\Haar(B(r_0))
\end{eqnarray}
and $\tdelta < \min(\k, \delta/10)$. Let $\tU = B(r_2)$. 

Suppose $i$ is large enough so that $\vol(M_i[\tU^2]) > (1-\tdelta^{10})\vol(M_i)$. Let $\phi \in \Map_{\avg}\left(M_i,X, \rho:\tU,\tdelta^{10} \right)$. It suffices to show $\rho^{M_i}(g \circ \phi, \phi \circ g) < \d$ for all $g \in U$. 

For $g \in \tU$, let $M_i[g]$ be the set of points $p\in M_i$ such that $g.p$ is well-defined. Also define
\begin{eqnarray*}
\tM_i(g)&=&\left\{ p \in M_i[g]:~ \rho(g \phi(p), \phi(g.p))<\tdelta^4\right\}\\
W &=& \left\{g \in \tU:~ \bvol(\tM_i(g)) > 1-\tdelta^4\right\}.
\end{eqnarray*}
Because $\phi$ is $(\tU,\tdelta^{10},\rho^{M_i})$-equivariant-on-average,
\begin{eqnarray*}
 \tdelta^{10} \Haar(\tU)  &>& \int_{\tU} \int_{M_i[g]} \rho(g \phi(p), \phi(g.p))~\dee\bvol(p)~\dee\Haar(g) \\
 &\ge& \int_{\tU \setminus W} \int_{M_i[g] \setminus \tM_i(g)} \rho(g \phi(p), \phi(g.p))~\dee\bvol(p)~\dee\Haar(g) \\
 &\ge& \tdelta^4 \int_{\tU \setminus W} \bvol(p)(M_i[g] \setminus \tM_i(g)) ~\dee\Haar(g) \\
  &\ge& \tdelta^4 \int_{\tU \setminus W} \tdelta^4 - \bvol(p)(M_i \setminus M_i[g]) ~\dee\Haar(g) \\
  &\ge& \Haar(\tU \setminus W)  \tdelta^4(\tdelta^4-\tdelta^{10}) >  \Haar(\tU \setminus W)  \tdelta^9
 \end{eqnarray*}
 
Therefore,
$$ \tdelta \Haar(\tU) > (1-\tdelta)\Haar(\tU \setminus W).$$
We use $\tU \supset B(r_0)$ and (\ref{E:tdelta1}) to obtain
\begin{eqnarray*}
\Haar(W \cap B(r_0)) &\ge&  \Haar(B(r_0)) - \Haar(\tU \setminus W)\\
&>&  \Haar(B(r_0)) - \tdelta\Haar(\tU) \ge (1-\tdelta)\Haar(B(r_0)).
 \end{eqnarray*}

Now suppose $g \in U$. By Lemma \ref{L:group}, there exist $h \in W \cap B(r_0)$ and $k \in W$ such that $g=hk$.  
By the triangle inequality,
\begin{eqnarray*}
&&\rho^{M_i}(g \circ \phi, \phi \circ g) = \rho^{M_i}(hk \circ \phi, \phi \circ hk) \\
&\le& \rho^{M_i}(hk \circ \phi, h\circ (\phi \circ k))  + \rho^{M_i}( h\circ (\phi \circ k), (\phi \circ h)\circ k )  + \rho^{M_i}((\phi \circ h)\circ k, \phi \circ hk).
\end{eqnarray*}

To estimate the first summand, note that, because $k \in W$, $\vol(\tM_i(k))>(1-\tdelta^4)\vol(M_i)$. Moreover, if $p \in \tM_i(k)$ then $\rho(k\phi(p), \phi(k.p)) < \tdelta^4 < \k$ which implies (by choice of $r_0$) that $\rho(hk\phi(p), h\phi(k.p))< \d/10$. Thus
$$  \rho^{M_i}(hk \circ \phi, h\circ (\phi \circ k))  \le \d/10 + \bvol(M_i \setminus \tM_i(k)) \le \d/10+ \tdelta^4 < 2\d/10.$$

To estimate the second summand, suppose $k.p \in \tM_i(h)$ (i.e. $p \in k^{-1}.\tM_i(h)$). Then $\rho(h\phi(k.p), \phi(h.k.p)) \le \tdelta^4$. Recall that volume is locally preserved by Lemma \ref{L:mp}. Since $k \in W$, it follows that
$$\vol(k^{-1}.\tM_i(h)) \ge \vol(\tM_i(h) \cap M[\tU]) \ge (1-\tdelta^4-\d/10)\vol(M_i).$$
Therefore,
$$\rho^{M_i}( h\circ (\phi \circ k), (\phi \circ h)\circ k ) \le 2\tdelta^4 +\d/10 \le 3\d/10.$$
For the last summand, suppose $p \in M_i[\tU^2]$. Then $h.k.p = hk.p$. Therefore,
$$\rho^{M_i}(  (\phi \circ h)\circ k , \phi \circ hk ) \le \bvol(M_i \setminus M_i[\tU^2]) \le \d/10.$$
Combining these estimates yields
$$\rho^{M_i}(g \circ \phi, \phi \circ g) \le 6\d/10<\d.$$
Since $g \in U$ is arbitrary, this proves $\phi \in \Map\left(M_i,X, \rho:U,\delta \right)$. 
\end{proof}

\subsection{Replacing pseudo-metrics with metrics}

Next we prove that pseudo-metrics can be replaced with metrics.
\begin{lem}\label{lem:metric}
Let $\rho$ be a uniformly continuous, generating, 1-bounded pseudo-metric on $X$, and $f:G \to (0,\infty)$ be a continuous function with $1 = \int f~\dee\Haar <\infty$.  Define $\rho_f$ by
$$\rho_f(x,y) : = \int \rho(gx,gy)f(g)~\dee\Haar(g).$$
Then $h_\Si\left(G, X, \rho\right) = h_\Si\left(G, X, \rho_f\right)$. Moreover, $\rho_f$ is a uniformly continuous $1$-bounded metric on $X$. 
\end{lem}

\begin{proof}
It is straightforward to check that $\rho_f$ is a $1$-bounded uniformly continuous pseudo-metric. To check that it is a metric, suppose $x,y \in X$ and $\rho_f(x,y)=0$. Since the integrand defining $\rho_f(x,y)$ is non-negative, this implies $\rho(gx,gy)f(g)=0$ for a.e. $g$. However $f>0$. Therefore $\rho(gx,gy)=0$ for a.e. $g$. In particular, $\rho(gx,gy)=0$ for all $g$ in a dense subset of $G$. Since $\rho$ and the action are continuous, the map $g \mapsto \rho(gx,gy)$ is continuous. Therefore, $\rho(gx,gy)=0$ for all $g \in G$. Since $\rho$ is generating, this implies $x=y$. Therefore $\rho_f$ is a metric. 


Next, we check the inequality $h_\Si\left(G,X, \rho\right) \le h_\Si\left(G,X, \rho_f\right)$. To see this, let $U_f \subset G$ be a pre-compact neighborhood of the identity and $\d_f>0$.

\noindent {\bf Claim 1}. There exist $\d$ and a pre-compact open $W \subset G$ satisfying

\begin{itemize}
\item $U_f \subset W$;
\item $0<\d <\d_f/10$;
\item if $x,y \in X$ and $\Haar(\{g\in W:~\rho(gx,gy)>\d\})<\d\Haar(W)$ then $\rho_f(x,y)< \d_f/2$.
\end{itemize}

\begin{proof}[Proof of Claim 1]
Choose $W \supset U_f$ large enough so that
$$\int_{G \setminus W} f(g)~\dee\Haar(g) < \d_f/10.$$
Choose $\d>0$ small enough so that $\d <\d_f/10$ and if $V \subset W$ is any set with Haar measure $\le \d\Haar(W)$ then  
$$\int_V f(g)~\dee\Haar(g) < \d_f/10.$$

We claim the last condition is now satisfied. To check this, let $x,y \in X$ and $V = \{g\in W:~\rho(gx,gy)>\d\}$. Suppose $\Haar(V)<\d \Haar(W)$. Let $I(g) = \rho(gx,gy)f(g)$. Because $\rho$ is $1$-bounded, 
\begin{eqnarray*}
\rho_f(x,y) &=& \int_{G \setminus W}I(g) ~\dee\Haar(g) +\int_{W \setminus V} I(g) ~\dee\Haar(g) + \int_{V} I(g) ~\dee\Haar(g)\\
&\le& \d_f/10 +  \d + \d_f/10   \le 3\d_f/10 < \d_f/2.
\end{eqnarray*}
\end{proof}

Let $U=WU_f$. 

\noindent {\bf Claim 2}.
$$\Map\left(M_i,X,\rho:U,\delta^{10}\right) \subset \Map_{\avg}\left(M_i,X,\rho_f : U_f,\delta_f\right).$$

\begin{proof}[Proof of Claim 2]
Fix $\phi \in \Map\left(M_i,X, \rho:U,\delta^{10}\right)$. It suffices to show
$$\d_f \Haar(U_f) > \int_{U_f}  \rho^{M_i}_f(g \circ \phi, \phi\circ g)~\dee\Haar(g).$$
For $g \in U_f$ and $p \in M_i$, let
\begin{eqnarray*}
W(g,p) &=& \left\{h \in W:~\rho(h g\phi(p), h \phi(g.p))\le \d^4 \right\}\\
\tM_i(g) &=& \left\{p \in M_i:~  \Haar(W(g,p)) \ge (1-\d^4) \Haar(W)\right\}.
\end{eqnarray*}

Because $\phi$ is $(U,\d^{10},\rho^{M_i})$-equivariant, $g\in U_f$ and $WU_f = U$,
\begin{eqnarray*}
&&\d^{10} \Haar(W)  \ge \int_W \rho^{M_i}(h g \circ \phi, h \circ (\phi \circ g))~\dee\Haar(h)   \\
&\ge&\bvol(M_i \setminus M_i[g])\Haar(W) + \int_{M_i[g] \setminus \tM_i(g)}  \int_{W \setminus W(g,p)}  \rho(h g\phi(p), h \phi(g.p))~\dee\Haar(h)~\dee\bvol(p)   \\
&\ge&\bvol(M_i \setminus M_i[g])\Haar(W) + \bvol(M_i[g] \setminus \tM_i(g)) \d^{8} \Haar(W). 
\end{eqnarray*}
Solve for $\bvol(\tM_i(g))$ to obtain
\begin{eqnarray*}
\bvol(\tM_i(g)) &\ge& 1 + (1-\d^{-8})\bvol(M_i[g]) - \d^2 \ge 1-\d^2.
\end{eqnarray*}
By Claim 1, if $p \in \tM_i(g)$ then $\rho_f(g \phi(p), \phi(g.p)) < \d_f/2$. Because $\rho_f$ is $1$-bounded, 
$$\int_{U_f} \rho_f^{M_i}( g\circ \phi, \phi \circ g) ~\dee\Haar(g) < (\d_f/2 + \bvol(M_i \setminus \tM_i(g)) )\Haar(U_f) < \d_f \Haar(U_f)$$
 as required.

\end{proof}

By uniform continuity of the $G$-action and continuity of $f$, for any $\eps'>0$ there is an $\eps'_f>0$ such that if $\rho(x,y)>\eps'$ then $\rho_f(x,y)>\eps'_f$. We integrate this inequality over a local $G$-space $M_i$ next.

\noindent {\bf Claim 3}.
For every $\eps>0$ there exists $\eps_f>0$ such that if $\phi, \psi \in \Map(M_i,X)$ satisfy $\rho^{M_i}(\phi,\psi)>\eps$ then $\rho^{M_i}_f(\phi,\psi)>\eps_f$. 

\begin{proof}[Proof of Claim 3]
$\rho^{M_i}(\phi,\psi)>\eps$ implies, via Markov's inequality, that for any $t>0$
\begin{eqnarray*}
\bvol(\{ p \in M_i:~ 1-\rho(\phi(p),\psi(p)) \ge t \}) &\le& \frac{ 1-\rho^{M_i}(\phi, \psi) }{t}< \frac{1-\eps}{t}.
\end{eqnarray*}
Set $t=\sqrt{1-\eps}$. By the paragraph above Claim 3 (with $\eps'=1-t$), we obtain $\eps'_f$ such that if $\rho(x,y)>\eps'$ then $\rho_f(x,y)>\eps'_f$. So if $\rho^{M_i}(\phi,\psi)>\eps$ then $\rho_f(\phi(p),\psi(p))> \eps'_f$ on a set of measure at least $1 - \frac{1-\eps}{t}$. Thus
$$\rho^{M_i}_f(x,y)> \eps'_f\left(1 - \frac{1-\eps}{t}\right) =\eps'_f(1- \sqrt{1-\eps}).$$

\end{proof}

It follows from Claims 2 and 3 that
$$\sep_\epsilon\left(\Map\left(M_i,X,\rho:U,\delta^{10}\right)\right) \le \sep_{\epsilon_f}\left(\Map_{\avg}\left(M_i,X,\rho_f:U_f,\delta_f\right)\right).$$
We can now take the logarithm of both sides, divide by $\vol(M_i)$, take the limsup as $i\to\infty$, the inf over $U,\d$, the inf over $U_f,\d_f$, the sup over $\eps_f$ and then the sup over $\eps$ to obtain 
$$h_\Si\left(G,X, \rho\right) \le h_{\Si,\avg}\left(G,X, \rho_f\right) = h_\Si\left(G,X, \rho_f\right)$$
 where we have used Proposition \ref{P:weakened} in the last equality.

To obtain the opposite inequality, we forget the previous choices of $U_f, \d_f, \d, W$ etc. The contrapositive of Claim 3 implies for all $\d>0$ there exist $\k>0$ such that if $\phi, \psi \in \Map(M_i,X)$ satisfy $\rho^{M_i}_f(\phi,\psi)<\k$ then $\rho^{M_i}(\phi,\psi)<\d$. Thus
\begin{eqnarray}\label{E:modelinclusion}
\Map\left(M_i,X,\rho:U,\delta\right) \supset \Map\left(M_i,X,\rho_f:U_f,\delta_f\right)
\end{eqnarray}
whenever $U \subset U_f$ and $\d_f \le \k$.  


\noindent {\bf Claim 4}. Let $U,\d,\eps_f$ be given. Then there exist $U_f, \d_f,\eps, I$ such that if $i>I$ then $\forall \phi, \psi \in \Map\left(M_i,X,\rho:U_f,\delta_f\right)$, if $\rho^{M_i}(\phi,\psi)\le \eps$ then $\rho^{M_i}_f(\phi,\psi)\le \eps_f$. 

\begin{proof}[Proof of Claim 4]
First we choose the parameters $U_f, \d_f,\eps,I$. 

By Claim 1, there exist $\eta>0$ and $W \subset G$ satisfying $U \subset W$, and if $x,y \in X$ and $\Haar(\{g\in W:\rho(gx,gy)>\eta\})<\eta \Haar(W)$ then $\rho_f(x,y)<\eps_f/10$. 

Let $U_f \supset W$ be a pre-compact open set. Choose $\d_f, \eps>0$ so that
$$2\d_f+11\eps/10< \eta^2\eps_f/10.$$ 
Let $i$ be large enough so that $\vol(M_i\setminus M_i[W])<\eps/10$. 

Let $\phi, \psi \in \Map\left(M_i,X,\rho:U_f,\delta_f\right)$ and suppose $\rho^{M_i}(\phi,\psi)\le \eps$. It suffices to show $\rho^{M_i}_f(\phi,\psi)\le \eps_f$. To do this, let
$$M(\eta) = \left\{p \in M_i:~ \Haar(\{g\in W:~\rho(g\phi(p),g\psi(p))>\eta\})<\eta\Haar(W)\right\}.$$
The constant $\eta$ was chosen so that if $p \in M(\eta)$ then $\rho_f(\phi(p),\psi(p))<\eps_f/10$. So it suffices to prove $\vol(M(\eta))>(1-\eps_f/10)\vol(M_i).$ We will obtain this estimate using what we know about $\phi$ and $\psi$ as follows.

By the triangle inequality,
\begin{eqnarray*}
&&\int_W \rho^{M_i}(g \circ \phi, g\circ \psi)~\dee\Haar(g) \\
&\le& \int_W \left[\rho^{M_i}(g \circ \phi, \phi \circ g) +\rho^{M_i}(\phi \circ g, \psi \circ g) +\rho^{M_i}(\psi \circ g, g\circ \psi)\right]~ \dee\Haar(g).
\end{eqnarray*}
Because $W \subset U_f$,
$$ \int_W \left[\rho^{M_i}(g \circ \phi, \phi \circ g)  +\rho^{M_i}(\psi \circ g, g\circ \psi)\right]~ \dee\Haar(g) \le 2\d_f \Haar(W).$$
Because $\rho^{M_i}(\phi,\psi)\le \eps$ and $\vol(M_i\setminus M_i[W])<\eps/10$,
\begin{eqnarray*}
&& \int_W \rho^{M_i}(\phi \circ g, \psi \circ g) ~ \dee\Haar(g) \\
&\le& \bvol(M_i \setminus M_i[W])\Haar(W) + \int_W \int_{M_i[W]} \rho( \phi(g.p), \psi(g.p))~\dee\bvol(p)~\dee\Haar(g)\\
 &\le& (\eps +\eps/10)\Haar(W)
\end{eqnarray*}
where the last inequality uses Lemma \ref{L:mp} (that the measure on $M_i$ is locally invariant under the partial action).
Combine this with the previous inequalities to obtain
$$\int_W \rho^{M_i}(g \circ \phi, g\circ \psi)~\dee\Haar(g) \le (2\d_f+11\eps/10)\Haar(W).$$
On the other hand, 
\begin{eqnarray*}
\int_W \rho^{M_i}(g \circ \phi, g\circ \psi)~\dee\Haar(g) & \ge & \int_{M_i\setminus M(\eta)} \int_W \rho(g \phi(p), g\psi(p))~\dee\Haar(g) ~\dee\bvol(p)\\
&\ge& \bvol(M_i \setminus M(\eta)) \Haar(W) \eta^2.
\end{eqnarray*}
Thus 
$$\bvol(M(\eta)) \ge \left(1- \frac{2\d_f+11\eps/10}{\eta^2}\right) > 1-\eps_f/10.$$ 
as required.

\end{proof}

It follows from (\ref{E:modelinclusion}) and Claim 4 that given $U,\d,\eps_f$ there exist $U_f,\d_f,\eps$ and $I$ such that if $i>I$ then
$$\sep_\epsilon\left(\Map\left(M_i,X,\rho:U,\delta\right)\right) \ge \sep_{\epsilon_f}\left(\Map\left(M_i,X,\rho:U_f,\delta_f\right)\right).$$
We can now take the logarithm of both sides, divide by $\vol(M_i)$, take the limsup as $i\to\infty$, the inf over $U_f,\d_f$, the inf over $U,\d$, the sup over $\eps$ and then the sup over $\eps_f$ to obtain $h_\Si\left(G,X, \rho\right) \ge h_\Si\left(G,X, \rho_f\right)$.

\end{proof}

\subsection{Proof of Theorem \ref{thm:top}}

\begin{proof}[Proof of Theorem \ref{thm:top}]

Let $\Phi:X \to Y$ be a uniform conjugacy between $G$ actions $G \cc (X,\rho_X)$ and $G \cc (Y,\rho_Y)$. By Lemma \ref{lem:metric}, we may assume $\rho_X$ and $\rho_Y$ are metrics.  

Given a local $G$-space $M$, let $\Phi_M:\Map(M,X) \to \Map(M,Y)$ be composition with $\Phi$. This means $\Phi_M(\phi) = \Phi \circ \phi$. Because $\Phi$ is uniformly continuous, the family of maps $\{\Phi_M\}_M$ is uniformly equicontinuous. To be precise this means: for every $\eps_Y>0$ there exists $\eps_X>0$ such that if $\rho^{M}_X(\phi,\psi)<\eps_X$ then $\rho^{M}_Y(\Phi\circ \phi, \Phi\circ \psi)<\eps_Y$ (and vice versa). Moreover, $\eps_X$ does not depend on $M$.  

Let $U_Y \subset G$ be a pre-compact neighborhood of the identity and $\d_Y>0$. By the previous paragraph, there exists $\d_X>0$ such that if $U_X \supset U_Y$ then $\Phi_M$  embeds $\Map\left(M,X,\rho_X:U_X,\delta_X\right)$ into $\Map\left(M,Y,\rho_Y:U_Y,\delta_Y\right).$ 

Moreover, for every $\eps_X>0$ there exists $\eps_Y>0$ such that if $\phi, \psi \in \Map(M_i,X)$ satisfy $\rho^{M_i}_X(\phi,\psi)>\eps_X$ then $\rho^{M_i}_Y(\Phi \circ \phi,\Phi \circ \psi)>\eps_Y$. Thus if $\cS \subset \Map\left(M,X,\rho_X:U,\delta\right)$ is $(\rho^{M_i}_X,\eps_X)$-separated, then $\Phi_M(\cS)$ is $(\rho^{M_i}_Y,\eps_Y)$-separated.

As in the proof of Lemma \ref{lem:metric}, we now have
$$\sep_{\epsilon_X}\left(\Map\left(M_i,X,\rho_X:U_X,\delta_X\right)\right) \le \sep_{\epsilon_Y}\left(\Map\left(M_i,Y,\rho_Y:U_Y,\delta_Y\right)\right)$$
for all indices $i$. Now we take the logarithm of both sides, divide by $\vol(M_i)$, take the limsup as $i\to\infty$, the inf over $U_X,\d_X$, the inf over $U_Y,\d_Y$, the sup over $\eps_Y$ and then the sup over $\eps_X$ to obtain $h_\Si\left(G,X, \rho_X\right) \le h_\Si\left(G,Y, \rho_Y\right)$. By symmetry, the theorem is proved.

\end{proof}

\section{Preliminaries to measure entropy theory}\label{S:ptmet}

\subsection{Space of probability measures}

Let $X$ be a Polish space and $\Prob(X)$ denote the space of all Borel probability measures on $X$ with the weak* topology. This is the smallest topology such that for every bounded continuous function $f:X \to \R$ the map $\mu \mapsto \int f~\dee\mu$ is continuous on $\Prob(X)$. If $X$ is compact then $\Prob(X)$ is also compact by the Banach-Alaoglu Theorem. 

If $G \cc X$ is a jointly continuous action then let $\Prob_G(X) \subset \Prob(X)$ be the subspace of $G$-invariant Borel probability measures. This space is closed in $\Prob(X)$.

\subsection{Measure conjugacy}\label{S:mc}
In this paper, all probability spaces are standard. We denote such a space by $(X,\cB_X,\mu)$ or just $(X,\mu)$, leaving the sigma-algebra $\cB_X$ implicit. 

An action of $G$ on a space $X$ equipped with a sigma-algebra $\cB_X$ is {\bf measurable} if the action map $G \times X \to X$ defined by $(g,x) \mapsto gx$ is measurable. 

Suppose $G \cc (X,\cB_X,\mu)$ and $G \cc (Y,\cB_Y,\nu)$ are measurable actions on standard probability spaces. A {\bf measure-conjugacy} is a measurable map $\Phi:X \to Y$ with measurable inverse $\Phi^{-1}:Y \to X$ such that $\Phi_*\mu = \nu$ (where $\Phi_*\mu$ is the measure on $Y$ defined by $\Phi_*\mu(E) = \mu(\Phi^{-1}(E))$) and $\Phi(gx)=g \Phi(x)$ for $g\in G$ and $x\in X$.

\section{Measure sofic entropy}\label{S:mse}

The goal of this section is to define measure sofic entropy and prove that it is a measure-conjugacy invariant. The definition is similar to that of topological sofic entropy with the exception that the notion of `good map' must be restricted to maps which nearly take the normalized volume on a local $G$-space $M_i$ to the target measure $\mu$ on $X$. 

To make this precise, let $(X,d_X)$ be a Polish metric space, $\Prob(X)$ denote the space of Borel probability measures on $X$ with the weak* topology, $G \cc X$ a uniform action and $\mu \in \Prob(X)$ a $G$-invariant Borel probability measure on $X$. We will denote the action by the triple $(G,X,\mu)$. Also let $\rho:X\times X \to \R$ be a uniformly continuous, 1-bounded, pseudo-metric. 




\begin{defn}
Let $U\subset G$ be a pre-compact neighborhood of identity, $\cO \subset \Prob(X)$ an open neighborhood of $\mu$, $\eps >0$. Suppose that $M$ is a local $G$-space with finite measure $\vol$. As before we let $\bvol$ denote the normalized volume on $M$, so that $\bvol(M') = \frac{\vol(M')}{\vol(M)}$ for any $M' \subset M$. Denote by $\Map\left(M,X, \rho: U,\delta, \cO \right)$ the set of all maps $\phi \in \Map\left(M,X, \rho: U,\delta \right)$ such that $\phi_*\bvol \in \cO$. 

\end{defn}


\begin{defn}\label{MapEntropy}
 Let $\Sigma=\{M_i\}$ be a sofic approximation to $G$.  Define
\begin{eqnarray*}
h_\Si\left(G,X,\mu, \rho\right)= \lim_{\eps \searrow 0} \inf_\cO \inf_{U\subset G}\inf_{\d>0}\limsup_{i\to \infty}\frac{1}{\ovol\left(M_i\right)}\log \sep_\epsilon\left(\Map\left(M_i,X,\rho:U,\delta,\cO\right)\right)
\end{eqnarray*}
where the second infimum is over all pre-compact neighborhoods $U$ of the identity in $G$ and the  first infimum is over all open neighborhoods $\cO$ of $\mu$ in $\Prob(X)$.\end{defn}


The goal of this section is to prove:

\begin{thm}\label{thm:measure}
Suppose $(G,X,\mu)$ and $(G,Y,\nu)$ are uniformly continuous, pmp actions and $\rho_X,\rho_Y$ are 1-bounded, uniformly continuous, generating, pseudo-metrics on $X$, $Y$ respectively.  If $(G,X,\mu)$ and $(G,Y,\nu)$ are measurably conjugate then $h_\Si(G,X,\mu,\rho_X)=h_\Si(G,Y,\nu,\rho_Y)$. 
\end{thm}

\begin{defn}
After this theorem has been proven we will write $h_\Si(G,X,\mu) = h_\Si(G,X,\mu,\rho_X)$ where $\rho_X$ is any choice of 1-bounded, uniformly continuous, generating, pseudo-metric. This is the {\bf measure $\Si$-entropy} of the action $G \cc (X,\mu)$. It may depend on the choice of sofic approximation $\Si$ and it may take on the value $-\infty$ (this occurs whenever $\Map(M_i,X,\rho:U,\delta,\cO)$ is empty for some choice of $U,\d,\cO$ and all sufficiently large $i$). 
\end{defn}

\begin{remark}
If $G$ is a countable group then the definition of $h_\Si(G,X,\mu,\rho)$ reduces in a straightforward way to the standard definition of measure sofic entropy, as presented in \cite{MR3616077} for example. 
\end{remark}

The first step of the proof is to show we can assume $\rho$ is a metric without loss of generality.

\begin{lem}\label{lem:metric2}
If $\rho$ is a uniformly continuous, generating, 1-bounded, pseudo-metric for the action $G \cc X$ then there exists a uniformly continuous, 1-bounded metric $\rho'$ such that $h_\Si\left(G,X,\mu, \rho\right)= h_\Si\left(G,X,\mu, \rho'\right)$.
\end{lem}

\begin{proof}
The proof is nearly identical to the proof of Lemma \ref{lem:metric}. Details are left to the  reader.
\end{proof}
For the rest of this section, we fix a uniformly continuous, 1-bounded metric $\rho$ on $X$. 




\subsection{Entropy through measure-preserving microstates}

The next step is to show we can restrict to maps $\phi:M_i \to X$ which take the normalized volume $\bvol$ to $\mu$ {\em exactly}. This will be used in the proof of Theorem \ref{thm:measure}.

\begin{defn}
Let $\Map_{\textrm{mp}}\left(M_i, X,\rho:U,\delta\right)$ be the set of all $\phi \in \Map\left(M_i, X,\rho:U,\delta\right)$ such that $\phi_*\bvol=\mu$. Informally, elements of $\Map_{\textrm{mp}}\left(M_i, X,\rho:U,\delta\right)$ are called {\bf measure-preserving microstates}. We let
\begin{eqnarray*}
h^{\textrm{mp}}_\Si\left(G,X,\mu, \rho\right)= \lim_{\eps \searrow 0} \inf_{U\subset G}\inf_{\d>0}\limsup_{i\to \infty}\frac{1}{\ovol\left(M_i\right)}\log \sep_\epsilon\left(\Map_{\textrm{mp}}\left(M_i,X,\rho:U,\delta\right)\right)
\end{eqnarray*}
denote the corresponding entropy.
\end{defn}

\begin{prop}\label{prop:abs-cont}
Suppose that for every $i$, $\vol_{M_i}$ has no atoms. Then measure entropy can be computed using measure-preserving microstates. More precisely,
\begin{eqnarray*}
h_\Si\left(G,X,\mu, \rho\right)= h^{\textnormal{mp}}_\Si\left(G,X,\mu, \rho\right).
\end{eqnarray*}
\end{prop}

To prove the nontrivial inequality, we show that if $\phi:M \to X$ maps $\bvol$ into a given neighborhood of $\mu$, then there is another map $\psi:M\to X$, close to $\phi$, which maps $\bvol$ to $\mu$. 

\begin{lem}\label{L:perturb0}
Suppose $(M,\vol)$ is a nonatomic finite measure space and $\epsilon>0$. Then there exists an open neighborhood $\cO$ of $\mu$ in $\Prob(X)$ such that for any map $\phi:M \to X$ with  $\phi_*\bvol \in \cO$ there exists  $\psi:M \to X$ such that $\rho^{M}(\phi,\psi)<\epsilon$ and $\psi_*\bvol = \mu$. 
\end{lem}

\begin{proof}
Recall that if $Y \subset X$ then $\partial Y = \overline{Y} \cap \overline{X\setminus Y}$ is the {\bf boundary} of $Y$. Moreover, $Y$ is a {\bf continuity set} if $\mu(\partial Y)=0$. It is a standard fact that continuity sets form a dense algebra of the measure algebra. In other words, the collection of continuity sets is closed under finite intersections and unions, and for every $\delta>0$ and Borel set $Y \subset X$ there exists a continuity set $Y' \subset Y$ such that $\mu(Y \vartriangle Y')<\delta$ where $\vartriangle$ denotes symmetric difference. 

Let $\cP=\{P_1, P_2, \ldots \}$ be a partition of $X$ into continuity sets such that each $P_i$ has diameter $<\epsilon/10$ and positive measure. Let $n$ be large enough so that $\mu(\cup_{i=1}^n P_i) > 1-\eps/10$. 

Let $\cO$ be the set of all $\mu' \in \Prob(X)$ such that 
$$|\mu'(P_i) - \mu(P_i)| < \frac{\epsilon}{10} \mu(P_i)$$
for all $1\le i \le n$. By the portmanteau Theorem, $\cO$ is an open neighborhood of $\mu$.

Let $\phi:M \to X$ be a map such that $\phi_*\bvol \in \cO$. It now suffices to construct a map $\psi$ with $\rho^{M}(\phi,\psi)<\epsilon$ and $\psi_*\bvol = \mu$.

Let $Q_i = \phi^{-1}(P_i)$. Let $Q'_i \subset Q_i$ be a subset with 
$$\bvol(Q'_i) = \min( \bvol(Q_i), \mu(P_i)).$$
Let $Q'_0 = M \setminus \cup_{i=1}^n Q'_i$. 

Next, choose a map $\psi:M \to X$ satisfying
\begin{enumerate}
\item $\psi_*\bvol = \mu$,
\item for all $p \in Q'_i$, $\psi(p) \in P_i$.
\end{enumerate}
To see that such a map exists, define a measure $\bvol_i$ on $Q'_i$ by $\bvol_i(E) = \bvol(E \cap Q'_i)$. Because $\bvol(Q'_i) \le \mu(P_i)$ and $\bvol_i$ is non-atomic, there exists a map $\psi_i: Q'_i \to P_i$ such that $\psi_{i*}\bvol_i$ is absolutely continuous to $\mu$ and
$$\frac{d\psi_{i*}\bvol_i}{d\mu}(x)\le 1$$
for all $x \in P_i$. 

Because $\bvol$ is non-atomic, there is also a map $\psi_0:Q'_0 \to X$ such that the pushforward measure $\psi_{0*}\bvol$ is absolutely continuous to  $\mu$ and the Radon-Nikodym derivative satisfies
$$\frac{d\psi_{0*}\bvol}{d\mu}(x)= 1- \sum_{i\ge 1} \frac{d\psi_{i*}\bvol_i}{d\mu}(x)$$
for all $x \in X$. We can define $\psi:M \to X$ by setting $\psi(p)=\psi_i(p)$ for all $p \in Q'_i$ and all $i=0,1,\ldots, n$. 

If $p \in Q'_i$ for some $1 \le i\le n$ then both $\psi(p)$ and $\phi(p)$ are in $P_i$ which has diameter $<\eps/10$ and therefore $\rho(\psi(p),\phi(p))<\epsilon/10$. On the other hand, $\rho$ is 1-bounded. Combining these facts we estimate:
\begin{eqnarray*}
\rho^{M}(\psi, \phi) &=&  \int_{Q_0'} \rho(\psi(p),\phi(p))~\dee\bvol(p) + \sum_{i= 1}^n \int_{Q_i'} \rho(\psi(p),\phi(p))~\dee\bvol(p)\\
&< &  \bvol(Q_0') + \bvol(M\setminus Q_0')\epsilon/10 \le \bvol(Q_0') +\epsilon/10.
\end{eqnarray*}
On the other hand, since $\phi_*\bvol \in \cO$, we have $\bvol(Q_i) = \phi_*\bvol(P_i) \ge (1-\eps/10)\mu(P_i)$ for $1\le i \le n$. So $\bvol(Q'_i)\ge (1-\eps/10)\mu(P_i)$. Adding this up results in
$$\bvol\left(\cup_{i=1}^n Q'_i\right) \ge (1-\eps/10)\mu\left(\cup_{i=1}^n P_i\right) \ge (1-\eps/10)^2 > 1-2\eps/10.$$
So  $\bvol(Q'_0) \le 2\eps/10$. Combined with the previous estimate, this gives
$\rho^{M}(\psi, \phi) \le 3\eps/10 < \eps$ which completes the proof.

\end{proof}

The next lemma shows that a map $\psi$ obtained by perturbing a good map $\phi$ stays good, although with a slightly larger error depending only on the size of the perturbation and the quality of the sofic approximation.

\begin{lem}\label{L:perturb1}
For every pre-compact open neighborhood $U$ of the identity, there exists a function $f_U:(0,\infty)^2 \to \R$ satisfying:
\begin{enumerate}
\item if $M$ is a $(U,\kappa)$-sofic approximation to $G$, $\phi \in \Map(M,X,\rho:U,\d)$, $\psi\in \Map(M,X)$ and $\rho^{M}(\phi,\psi)<\epsilon$ then $\psi \in \Map(M,X,\rho:U, \d + f_U(\k,\eps))$;
\item $\lim_{\k,\eps \searrow 0} f_U(\kappa,\epsilon) =0$. 
\end{enumerate}
\end{lem}

\begin{proof}
Suppose $M,\phi,\psi, U, \k, \eps$ are as in the statement. Let $N=\{p\in M:~ \rho(\psi(p), \phi(p))<\sqrt{\epsilon}\}$. Then $\vol(N) \ge (1-\sqrt{\epsilon})\vol(M)$ by Markov's inequality. So for any $g \in G$,
\begin{eqnarray*}
\rho^{M}(g \circ \psi, g\circ \phi) &\le &\bvol(M\setminus N) + \int_N \rho(g\psi(p),g\phi(p))~\dee\bvol(p) \\
&\le &\sqrt{\epsilon} + \textrm{m.o.c.}[g:\sqrt{\epsilon}]
\end{eqnarray*}
where the modulus of continuity is defined by 
$$\textrm{m.o.c.}[g:\sqrt{\epsilon}]=\sup_{x,y \in X}\big\{ \rho(gx,gy):~\rho(x,y) \le \sqrt{\epsilon}\big\}.$$

Fix $g \in U$. Then $g.p$ is well-defined for $p\in M[U]$. Thus.
\begin{eqnarray*}
\rho^{M}( \psi \circ g, \phi \circ g) &\le &\bvol(M\setminus M[U])+\int_{M[U]} \rho(\psi(g. p),\phi(g. p))~\dee\bvol(p) \\
&\le &\kappa+\rho^{M}(\psi,\phi) \le \kappa+\eps.
\end{eqnarray*}

So
\begin{eqnarray*}
\rho^{M}(g \circ \psi, \psi \circ g) &\le& \rho^{M}(g  \circ\psi, g \circ\phi) + \rho^{M}(g \circ\phi, \phi  \circ g) + \rho^{M}(\phi  \circ g, \psi \circ g)\\
&\le & \sqrt{\epsilon} + \textrm{m.o.c.}[g:\sqrt{\epsilon}] + \delta + \kappa+ \epsilon \le f_U(\k,\eps) +\d
\end{eqnarray*}
where $f_U(\k,\eps) = \sqrt{\epsilon} + \eps+\kappa+\sup_{g\in U} \textrm{m.o.c.}[g:\sqrt{\epsilon}]$. This satisfies $\lim_{\k,\eps \searrow 0} f_U(\kappa,\epsilon) =0$ because $U$ is pre-compact and the action is uniformly continuous.
\end{proof}

\begin{proof}[Proof of Proposition \ref{prop:abs-cont}]
Clearly
$$\Map_{\textrm{mp}}\left(M_i,X,\rho:U,\delta\right) \subset \Map\left(M_i,X,\rho:U,\delta,\cO\right)$$
(for any pre-compact $U \subset G$, open neighborhood $\cO \ni \mu$ and $\delta>0$) which implies $h^{\textrm{mp}}_\Si\left(G,X,\mu, \rho_{X}\right) \le h_\Si\left(G,X,\mu, \rho_{X}\right)$. 

To prove the other inequality, let $\epsilon,\eta>0$ be constants. By Lemma \ref{L:perturb0}, there exists a neighborhood $\cO \subset \Prob(X)$ of $\mu$ such that if $\phi \in \Map\left(M_i,X,\rho:U,\delta,\cO\right)$ then there exists $\phi':M_i \to X$ with $\phi'_*\bvol = \mu$ and $\rho^{M_i}(\phi,\phi')<\eta$. Lemma \ref{L:perturb1} implies $\rho^{M_i}(g\circ \phi', \phi' \circ g) \le \d + f_U(i,\eta)$ for some function $f_U$ which tends to zero as $i\to\infty$ and $\eta \searrow 0$ while $U$ is held fixed. Thus 
$$\phi' \in \Map_{\textrm{mp}}\left(M_i,X,\rho:U,\delta+f_U(i,\eta)\right).$$

Now suppose $Y \subset \Map\left(M_i,X,\rho:U,\delta,\cO\right)$ is $(\rho^{M_i},\eps)$-separated. Let $Y' = \{\phi':~ \phi \in Y\}$. By the previous paragraph $Y' \subset \Map_{\textrm{mp}}\left(M_i,X,\rho:U,\delta+f_U(i,\eta)\right)$ is $\eps-2\eta$ separated. Thus
$$\sep_{\epsilon}\left(\Map\left(M_i,X,\rho:U,\delta,\cO\right)\right) \le \sep_{\epsilon-2\eta}\left(\Map_{\textrm{mp}}\left(M_i,X,\rho_{X}:U,\delta + f_U(i,\eta) \right)\right).$$
Now take the logarithm of both sides, divide by $\vol(M_i)$, take the limsup as $i\to\infty$ then the infimum over $\cO,\d$, then the infimum over $\eta$, then the infimum over $U$ and finally the supremum over $\eps$ to obtain $h^{\textrm{mp}}_\Si\left(G,X,\mu, \rho\right) \ge h_\Si\left(G,X,\mu, \rho\right)$.


\end{proof}

The next lemma shows how to avoid the non-atomic assumption from Proposition \ref{prop:abs-cont}.

\begin{lem}\label{lem:nonatomic}
Suppose $G$ is non-discrete and $M$ is a local $G$-space. Then $\vol_M$ has no atoms.
\end{lem}

\begin{proof}
Let $p \in M$ and let $f_p$ be a chart centered at $p$. By Proposition \ref{P:canonicalmeasure}, $\vol_M(K) = \Haar(f_p(K))$ for every Borel $K$ in the domain of $f_p$. In particular, $\vol_M(\{p\}) = \Haar(\{1_G\})$. Because $G$ is non-discrete, $\Haar(\{1_G\})=0$. Since $p$ is arbitrary, this implies $\vol_M$ has no atoms.
\end{proof}

\subsection{Proof of Theorem \ref{thm:measure}}


\begin{proof}[Proof of Theorem \ref{thm:measure}]
This Theorem is known in case $G$ is discrete (for example, see \cite{MR3616077}). So we assume $G$ is non-discrete. By Lemma \ref{lem:nonatomic} each $(M_i,\vol)$ is nonatomic. 

Suppose $(G,X,\mu)$ and $(G,Y,\nu)$ are measurably conjugate, uniformly continuous, pmp actions and $\rho_X,\rho_Y$ are 1-bounded, uniformly continuous, generating pseudo-metrics on $X$, $Y$ respectively. It suffices to prove $h_\Si(G,X,\mu,\rho_X)=h_\Si(G,Y,\nu,\rho_Y)$. By Lemma \ref{lem:metric2}, we may assume without loss of generality that $\rho_X$ and $\rho_Y$ are metrics rather than pseudo-metrics. By Proposition \ref{prop:abs-cont} we may consider measure-preserving microstates.

Let $\Phi:X \to Y$ be a measure-conjugacy. Let $\tau>0$ be a constant. Because $X$ and $Y$ are Polish spaces, $\mu$ and $\nu$ are Radon measures and so are inner regular. By Lusin's Theorem there exists a compact subset $L \subset X$ such that (a) $\mu(L)> 1-\tau$ and (b)  $\Phi$ restricted to $L$ is uniformly continuous.

\noindent {\bf Claim 1}. If $\phi \in \Map_{\textrm{mp}}(M, X, \rho_X: U, \delta)$ and $M$ is a $(U,\kappa)$-sofic approximation to $G$ then $\Phi \circ \phi \in \Map_{\textrm{mp}}(M, Y, \rho_Y:U, f_1(\kappa, \delta,\tau ))$ for some function $f_1$ satisfying
$$\lim_{\tau \searrow 0} \lim_{\d,\k \searrow 0} f_1(\k,\d,\tau) = 0.$$

\begin{proof}[Proof of Claim 1]
Since $\Phi$ is a measure-conjugacy and $\phi_*\bvol=\mu$, it is immediate that $(\Phi\circ \phi)_*\bvol=\nu$. So it suffices to show that $\rho^M_Y(\Phi\circ \phi \circ g,  g\circ \Phi \circ \phi ) <f_1(\k,\d,\tau)$ for all $g \in U$ (where $f_1$ is yet to be defined).  

Fix $g \in U$. Let 
\begin{eqnarray*}
N_1 &=& \{p\in M:~ \phi(g. p) \in L\},\\
N_2 &=& \{p \in M:~ g \phi(p) \in L\},\\
N_3 &=& \left\{p \in M:~ \rho_X(\phi(g.p) , g\phi(p)) < \sqrt{\delta}\right\}.
\end{eqnarray*}
Then
\begin{eqnarray*}
\bvol(N_1) &\ge& 1-\kappa-\tau\\
\bvol(N_2) &\ge& 1 - \tau,\\
\bvol(N_3) &\ge& 1-\sqrt{\delta}.
\end{eqnarray*}
The first inequality holds because $M$ is $(U,\k)$-sofic and $\mu(L)>1-\tau$ (this uses Lemma \ref{L:mp}). The second holds because $\phi_*\bvol = \mu$ so $\bvol(N_2) = \mu(g^{-1}L) = \mu(L)$. The third inequality holds by applying Markov's inequality to $\rho^{M}_{X}(\phi\circ g, g\circ \phi)<\d$. Therefore,
$$\bvol(N_1\cap N_2 \cap N_3) \ge 1 - \kappa - 2\tau - \sqrt{\delta}.$$
Because $\Phi$ is $G$-equivariant,
\begin{eqnarray*}
\rho^M_Y(\Phi\circ \phi \circ g,  g\circ \Phi \circ \phi ) &=& \rho^M_Y(\Phi\circ \phi \circ g,  \Phi \circ g \circ \phi ) \\
&\le & \kappa + 2\tau + \sqrt{\delta}  + \int_{N_1\cap N_2 \cap N_3} \rho_Y(\Phi (\phi (g.p)),  \Phi (g (\phi (p))))~\dee\bvol(p)  \\
&\le & \kappa + 2\tau+ \sqrt{\delta}  + \textrm{m.o.c.}[\Phi \resto L: \sqrt{\delta}]
\end{eqnarray*}
where the modulus of continuity is defined by
$$\textrm{m.o.c.}[\Phi \resto L: \sqrt{\delta}]=\sup_{x,y \in L}\{ \rho_Y(\Phi(x),\Phi(y)):~\rho_X(x,y) \le \sqrt{\delta}\}.$$
Set $f_1(\kappa,\delta,\tau) = \kappa + 2\tau+ \sqrt{\delta}  + \textrm{m.o.c.}[\Phi \resto L: \sqrt{\delta}]$ to finish the proof.

\end{proof}

We can improve upon Claim 1 as follows. Let $f'_1(\kappa,\delta)$ be the infimum of $f_1(\kappa,\delta,\tau)$ over all $\tau$. Claim 1 implies that if $f''_1(\kappa,\delta)$ is any constant strictly larger than $f'_1(\kappa,\delta)$ then 
 $$\Phi \circ \phi \in \Map_{\textrm{mp}}(M, Y, \rho^M_Y:U, f''_1(\kappa, \delta )).$$
We can choose $f''_1(\kappa,\delta)$ so that
\begin{eqnarray}\label{E:mse-limit}
\lim_{\k, \d \searrow 0} f''_1(\kappa,\delta) = 0. 
\end{eqnarray}


\noindent {\bf Claim 2}. Let $(M,\vol)$ be a finite measure space. For $i=1,2$, let $\phi_i:M\to X$ be measure-preserving (so $\phi_{i*}\bvol=\mu$). Suppose $\rho^M_Y(\Phi\phi_1,\Phi\phi_2)<\eps$. Then $\rho^M_X( \phi_1, \phi_2) < f_2(\epsilon,\tau)$
for some function $f_2$ satisfying
$$\lim_{\eps \searrow 0}  f_2(\eps,\tau) =2\tau.$$
\begin{proof}

Let $N_1 = \phi_1^{-1}(L)$ and $N_2 = \phi_2^{-1}(L)$ and
$$N_3 = \left\{p \in M:~ \rho_Y(\Phi\phi_1(p) , \Phi\phi_2(p)) < \sqrt{\epsilon} \right\}$$
Because $\phi_{1*}\bvol=\mu$, $\bvol(N_1)=\mu(L) \ge 1-\tau$. Similarly, $\bvol(N_2) \ge 1 - \tau$. Apply Markov's inequality to $\rho^M_Y(\Phi\phi_1,\Phi\phi_2)<\eps$ to obtain $\bvol(N_3) \ge 1-\sqrt{\epsilon}$. Thus
$$\bvol(N_1\cap N_2 \cap N_3) \ge 1 - 2\tau - \sqrt{\epsilon}.$$
Since $\rho_X$ is 1-bounded,
\begin{eqnarray*}
\rho^M_X( \phi_1,  \phi_2) &<& 2\tau + \sqrt{\epsilon} +  \int_{N_1\cap N_2 \cap N_3} \rho_X(\phi_1(p), \phi_2(p)) ~\dee\bvol(p)\\
&\le&  2\tau + \sqrt{\epsilon}  + \textrm{m.o.c.}[\Phi^{-1}\resto \Phi(L): \sqrt{\epsilon}]
\end{eqnarray*}
where the modulus of continuity is defined by
$$\textrm{m.o.c.}[\Phi^{-1}\resto \Phi(L): \sqrt{\epsilon}]=\sup_{x,y \in \Phi(L)}\{ \rho_X(\Phi^{-1}(x),\Phi^{-1}(y)):~\rho_Y(x,y) \le \sqrt{\eps}\}.$$
Set $f_2(\epsilon,\tau)= 2\tau + \sqrt{\epsilon}  + \textrm{m.o.c.}[\Phi^{-1}\resto \Phi(L): \sqrt{\epsilon}]$ to finish the claim.
\end{proof}

Let $\Omega$ be a covering of $\Map_{\textrm{mp}}\left(M_i, Y,\rho_Y:U, f''_1(\kappa_i,\delta) \right)$ by sets of $\rho_Y^{M_i}$-diameter $<\eps$. Then $\Phi^{-1}_*\Omega=\{\Phi^{-1}(\cY):~\cY \in \Omega\}$ is a covering of $\Map_{\textrm{mp}}\left(M_i, X,\rho_X:U,\delta\right)$ by sets of $\rho_X^{M_i}$-diameter $<f_2(\eps,\tau)$. Thus
\begin{eqnarray*}
\cov_{f_2(\epsilon, \tau)}\left(\Map_{\textrm{mp}}\left(M_i, X,\rho_X:U,\delta\right)\right) 
&\le& \cov_{\epsilon}\left(\Map_{\textrm{mp}}\left(M_i, Y,\rho_Y:U, f''_1(\kappa_i,\delta) \right)\right).
\end{eqnarray*}

We take the logarithm of both sides, divide by $\vol(M_i)$, take the limsup as $i\to\infty$, take the infimum over $U,\delta$  to obtain
\begin{eqnarray*}
&& \inf_{U\subset G} \inf_{\d>0} \limsup_{i\to \infty}\frac{1}{\ovol\left(M_i\right)} \cov_{f_2(\epsilon, \tau)}\left(\Map_{\textrm{mp}}\left(M_i, X,\rho_X:U,\delta\right)\right)\\
&& \le \inf_{U\subset G} \inf_{\d>0} \limsup_{i\to \infty}\frac{1}{\ovol\left(M_i\right)} \cov_{\epsilon}\left(\Map_{\textrm{mp}}\left(M_i, Y,\rho_Y:U, \d \right)\right).
\end{eqnarray*}
This uses (\ref{E:mse-limit}). Now take the limit as $\eps \searrow 0$ and then $\tau \searrow 0$ to obtain
$$h^{\textrm{mp}}_\Si(G,X,\mu, \rho_X) \le h^{\textrm{mp}}_\Si(G,Y,\nu,\rho_Y).$$
Since the opposite inequality follows in the same way, we have $h^{\textrm{mp}}_\Si(G,X,\mu,\rho_X)=h^{\textrm{mp}}_\Si(G,Y,\nu,\rho_Y)$. The theorem now follows from Proposition \ref{prop:abs-cont}.

\end{proof}

\section{Variational Principle}\label{S:vp}

The purpose of this section is to prove that topological sofic entropy is the sup over measure sofic entropies. To be precise, given an action $G \cc X$, let $\Prob_G(X) \subset \Prob(X)$ denote the space of $G$-invariant Borel probability measures. 

\begin{thm}\label{thm:variational}
Let $(X,\rho)$ be a compact metric space and $G\cc X$ a jointly continuous action. Then for any sofic approximation $\Si$ to $G$,
$$h_\Si(G, X) = \sup_{\mu \in \Prob_G(X)} h_\Si(G,X,\mu).$$
In particular, if there are no $G$-invariant probability measures on $X$ then $h_\Si(G, X) = -\infty$.  
\end{thm}

For future applications, it is nice to have a slightly more general version. If $A \subset \Prob(X)$ is any set, then define the entropy 
$$h_\Si(G,X,A) = \lim_{\eps \searrow 0} \inf_{\cO \supset A} \inf_{U \subset G} \inf_{\d > 0}\limsup_{n\to\infty} \vol(M_n)^{-1} \log \cov_\eps(\Map(M,X,\rho:U,\d,\cO))$$
where $\cO$ varies over all open sets containing $A$ and $U$ varies over all pre-compact open subsets $U$ of $G$. 

This definition is the same as the definition of measure entropy, with the exception that the open set $\cO$ is required to contain $A$ instead of $\mu$. In particular, if $A = \{\mu\}$ then $h_\Si(G,X,A) = h_\Si(G,X,\mu)$.  Therefore, the next theorem immediately implies Theorem \ref{thm:variational}.

\begin{thm}\label{thm:variational2}
Let $(X,\rho)$ be a compact metric space and $G\cc X$ a jointly continuous action. Then for any sofic approximation $\Si$ to $G$ and closed subset $A \subset \Prob(X)$
$$h_\Si(G,X,A) = \sup_{\mu\in A \cap \Prob_G(X)} h_\Si(G,X,\mu).$$
\end{thm}

The first step in the proof is showing that if $\phi:M \to X$ is a good enough topological model for the action $G \cc X$ then $\phi_*\bvol$ is almost $G$-invariant.

\begin{lem}\label{lem:equidistribution}
Let $G \cc X$ be as above. Let $\cO$ be an open neighborhood of $\Prob_G(X)$ in $\Prob(X)$. Then there exist a pre-compact open set $U \subset G$ and $\delta>0$ such that if $M$ is a $(U,\delta)$-sofic approximation to $G$ and $\phi \in \Map(M,X,\rho:U,\delta)$ then $\phi_*\bvol \in \cO$.

\end{lem}

\begin{proof}
Without loss of generality, we assume $\rho$ is $1$-bounded. 

If $f:X \to \R$ is continuous, $g\in G$ and $\eps>0$ then let 
$$\cO_{f,g,\eps}=\{\mu \in \Prob(X):~|\mu(f) - \mu(f \circ g)|<\epsilon\}.$$
Observe that $\Prob_G(X)$ is the intersection of $\cO_{f,g,\eps}$ over all such $f,g,\eps$. 

Let $\cO \subset \Prob(X)$ be an open set containing $\Prob_G(X)$. Note $\{\Prob(X) \setminus \overline{\cO_{f,g,\eps}}:~f,g,\eps\}$ is an open cover of $\Prob(X) \setminus \cO$. Because $\Prob(X) \setminus \cO$ is compact, there exist $f_1,\ldots, f_n$, $g_1,\ldots, g_n$ and $\eps_1,\ldots, \eps_n$ such that $\Prob(X)\setminus \cO$ is contained in $\cup_{i=1}^n \Prob(X) \setminus \overline{\cO_{f_i,g_i,\eps_i}}$. Therefore, $\cO$ contains $\cap_{i=1}^n \cO_{f_i,g_i,\eps_i}$.


So it suffices to prove the lemma in the special case in which $\cO=\cO_{f,g,\eps}$ for some $f:X\to \R$, $g\in G$ and $\eps>0$. 

Let $U\subset G$ be a pre-compact open set with $\{1_G, g, g^{-1}\} \subset U$. Choose $\delta>0$ so that $\rho(x,y)<\delta^{1/2}$ implies $|f(x)-f(y)|<\epsilon/2$ and $(4\d + 2\sqrt{\d})\|f\|_\infty  < \epsilon/2$. 

Suppose $M$ is a $(U,\delta)$-sofic approximation to $G$ and $\phi \in \Map(M,X,\rho:U,\delta)$. It suffices to show $\phi_*\bvol \in \cO$. Since $\phi_*\bvol(f)= \int f(\phi(p)) ~\dee\bvol(p)$ and $\phi_*\bvol(f\circ g)= \int f(g\phi(p)) ~\dee\bvol(p)$, it suffices to show that if 
$$\a = \left| \int f(\phi(p)) - f(g\phi(p)) ~\dee\bvol(p) \right|$$
then $\a < \eps.$ For this, let
\begin{eqnarray*}
\b &=& \left| \int_M f(\phi(p)) ~\dee\bvol(p) -\int_{M[g]} f(\phi(g.p)) ~\dee\bvol(p) \right|\\
\g &=& \left|\int_{M[g]}  f(\phi(g.p)) ~\dee\bvol(p) - \int_M f(g\phi(p)) ~\dee\bvol(p) \right|
  \end{eqnarray*}
where $M[g]$ is the set of all $p\in M$ such that $g.p$ is well-defined.   By the triangle inequality $\a \le \b + \g$. 

We estimate $\b$ first. By Lemma \ref{L:mp},
$$ \int_{M[U]} f(\phi(p)) ~\dee\bvol(p)  = \int_{g^{-1}.M[U]} f(\phi(g.p)) ~\dee\bvol(p).$$
Because $M$ is $(U,\d)$-sofic, $\bvol(M[U])\ge (1-\d)$. So 
$$\b  =  \left| \int_{M \setminus M[U]} f(\phi(p)) ~\dee\bvol(p)- \int_{M[g] \setminus g^{-1}.M[U]} f(\phi(g.p)) ~\dee\bvol(p) \right| \le 2\d \|f\|_\infty.$$

To estimate $\g$, let $M'$ be the set of all $p \in M[U]$ such that $\rho(\phi(g.p), g\phi(p)) < \delta^{1/2}$. Note $|f(g\phi(p)) - f(\phi(g.p))| <\eps/2$ for all $p \in M'$ by the choice of $\d$. So
\begin{eqnarray}\label{E1}
  \left| \int_{M'} f(\phi(g.p)) - f( g\phi(p))~\dee\bvol(p) \right| < \eps/2.
  \end{eqnarray}
By Markov's inequality (and because $M' \subset M[U]$),  $\bvol(M') \ge 1-\d-\sqrt{\d}$. So 
$$\g \le \eps/2 +   \left| \int_{M[g]\setminus M'} f(\phi(g.p))~\dee\bvol(p)  - \int_{M\setminus M'}f( g\phi(p))~\dee\bvol(p) \right|  \le \eps/2 + 2(\d + \sqrt{\d})\|f\|_\infty.$$
Thus
$$\a \le \b+\g\le \eps/2 + (4\d + 2\sqrt{\d})\|f\|_\infty < \eps.$$

\end{proof}

\begin{proof}[Proof of Theorem \ref{thm:variational2}]
The inequality 
$$h_\Si(G,X,A) \ge \sup_{\mu \in A \cap \Prob_G(X)} h_\Si(G,X,\mu)$$
is immediate. 

To prove the opposite inequality, we may assume without loss of generality that $h_\Si(G,X,A) \ne -\infty$. Fix $\eps>0$. Given $i\in \N$, $U \subset G$, $\d>0$ and an open neighborhood $\cO$ of $A$ in $\Prob(X)$, let 
$$S(i,U,\d,\cO) \subset \Map(M_i,X,\rho:U,\d,\cO)$$
be a maximal $(\rho^{M_i},\eps)$-separated subset. Define
$$h^\eps_\Si(G,X,A) = \inf_{\cO} \inf_U \inf_{\d>0} \limsup_{i \to\infty} \vol(M_i)^{-1} \log |S(i,U,\d,\cO)|.$$
By definition, 
$$h_\Si(G,X,A) = \lim_{\eps \searrow 0} h^\eps_\Si(G,X,A).$$
Since we assume $h_\Si(G,X,A) \ne -\infty$, it follows that for all $U,\d,\cO$ the set $S(i,U,\d,\cO)$ is non-empty for infinitely many $i$. 

For each $\phi \in \Map(M_i,X)$, let $\d_{\phi_*\bvol} \in \Prob(\Prob(X))$ be the Dirac mass concentrated on the pushforward $\phi_*\bvol$. Let
$$Q(i,U,\d,\cO) = |S(i,U,\d,\cO)|^{-1} \sum_{\phi \in S(i,U,\d,\cO)} \d_{\phi_*\bvol} \in \Prob(\Prob(X))$$
whenever $S(i,U,\d,\cO)$ is non-empty. 

For each $U,\d,\cO$, choose an increasing sequence $\{n_i\}_{i=1}^\infty$ such that 
\begin{eqnarray}\label{E:limsoupy}
\limsup_{i \to\infty} \vol(M_i)^{-1} \log |S(i,U,\d,\cO)| = \lim_{i \to\infty} \vol(M_{n_i})^{-1} \log |S(n_i,U,\d,\cO)|.
\end{eqnarray}

Because $X$ is compact, both $\Prob(X)$ and $\Prob(\Prob(X))$ are compact in the weak* topology (by the Banach-Alaoglu Theorem). So after passing to a further subsequence if necessary, we may assume
$$\lim_{i\to\infty} Q(n_i,U,\d,\cO) = Q(U,\d,\cO)$$
for some measure $Q(U,\d,\cO) \in \Prob(\Prob(X))$. Let $Q_*$ be an accumulation point of $Q(U,\d,\cO)$ as $U \nearrow G$, $\d\searrow 0$ and $\cO \searrow A$. 

By Lemma \ref{lem:equidistribution}, $Q_*(\Prob_G(X)) = 1$. By construction, $Q_*( A)=1$. So the support of $Q_*$ is contained in $A \cap \Prob_G(X)$. Let $\mu$ be in the support of $Q_*$. It suffices to show $h^\eps_\Si(G,X,\mu) \ge h^\eps_\Si(G,X, A)$. 

Let $\cO_\mu$ be an open neighborhood of $\mu$ in $\Prob(X)$. Then $Q_*(\cO_\mu)>0$. Also fix $U_0, \d_0, \cO_0$. By the portmanteau Theorem, there exist $U, \d, \cO$ satisfying
\begin{itemize}
\item $U \supset U_0$, $0<\d\le \d_0$, $A \subset \cO \subset \cO_0$, and
\item $Q(U,\d,\cO)(\cO_\mu) \ge Q_*(\cO_\mu)/2 >0$. 
\end{itemize}
Let $\{n_i\}_{i=1}^\infty$ be the increasing sequence mentioned above for $U,\d,\cO$.  By the portmanteau Theorem again, there exists $I$ such that if $i>I$ then $Q(n_i,U,\d,\cO)(\cO_\mu) \ge Q_*(\cO_\mu)/4>0$. Since
$$\{\phi \in S(n_i,U,\d,\cO):~ \phi_*\bvol \in \cO_\mu\} = S(n_i,U,\d,\cO) \cap \Map(M_{n_i},X,\rho:U,\d,\cO\cap \cO_\mu),$$
this means
$$|S(n_i,U,\d,\cO) \cap \Map(M_{n_i},X,\rho:U,\d,\cO\cap \cO_\mu)| \ge |S(n_i,U,\d,\cO)|Q_*(\cO_\mu)/4.$$
Thus 
\begin{eqnarray*}
\sep_\eps(\Map(M_{n_i},X,\rho:U,\d, \cO_\mu)) &\ge& \sep_\eps(\Map(M_{n_i},X,\rho:U,\d,\cO\cap \cO_\mu))\\
& \ge& |S(n_i,U,\d,\cO)|Q_*(\cO_\mu)/4.
\end{eqnarray*}
Take the logarithm of both sides, divide by $\vol(M_{n_i})$ and take the limsup as $i\to\infty$ then the infimum over $U,\d,\cO$ to obtain
$$ \inf_{U \subset G} \inf_{\d > 0}\limsup_{i\to\infty} \vol(M_i)^{-1} \log \sep_\eps(\Map(M_i,X,\rho:U,\d,\cO_\mu)) \ge h^\eps_\Si(G,X, A).$$
Here we have used (\ref{E:limsoupy}). Next take the infimum over $\cO_\mu$ to obtain
 $h^\eps_\Si(G,X,\mu) \ge h^\eps_\Si(G,X, A)$. 

\end{proof}

\section{Open Questions}\label{S:oq}

\begin{enumerate}




\item If $G$ is amenable and unimodular then it is sofic \cite{MR4555895}. In this case, does sofic entropy agree with the spatial entropy of Ornstein-Weiss \cite{OW87}? 

\item Does Avni's approach to entropy \cite{avni-2010} through cross-sections generalize to sofic entropy?



\item Let $G$ be a locally compact group with a sofic approximation $\Si$. The {\bf base entropy of $\Si$} is the infimum of $h_\Si(G,X,\mu)$ over all essentially free pmp actions $G \cc (X,\mu)$. We denote it by $h_\Si(G)$. If $G$ is amenable then we also let $h(G)$ be the infimum of spatial entropies of essentially free pmp actions. If item (1) above has a positive answer, then $h(G)=h_\Si(G)$. Ornstein and Weiss gave examples where $h(G)>0$ and even where $h(G) = +\infty$ \cite{OW87}. If $t \ge h_\Si(G)$ then does there exist an essentially free ergodic pmp action $G \cc (X,\mu)$ with $\Si$-entropy equal to $t$? Is there a group-theoretic characterization of which groups $G$ admit a sofic approximation with $h_\Si(G)>0$ or with $h_\Si(G)=+\infty$? If $\Si_1,\Si_2$ are two sofic approximations then is $h_{\Si_1}(G) = h_{\Si_2}(G)$?

\item If $G$ is an lcsc group with a sofic approximation $\Si$ and $G \cc (X,\mu)$ is a pmp action then we say this action has {\bf completely positive entropy} (CPE) with respect to $\Si$ if every nontrivial factor has positive $\Si$-entropy. Do Poisson point processes have completely positive entropy (CPE) with respect to every sofic approximation? If $G$ is countable then the answer is yes by \cite{kerr-cpe, MR3996050}. 

\item Let $G$ be a locally compact sofic group. Is there a class $\cB$ of pmp actions of $G$ which have properties similar to Bernoulli shifts? Ideally, such a class would have the following properties.
\begin{enumerate}
\item For every number $t \in (h_\Si(G),\infty)$ there is exactly one pmp action in $\cB$ (up to measure conjugacy) with sofic entropy $t$. Moreover,  this entropy does not depend on the choice of sofic approximation. 
\item Each action in this class has CPE (completely positive entropy). This means every nontrivial factor of each action has positive entropy.
\item If $G$ is amenable and $\a_1,\a_2$ are two actions in $\cB$ and the entropy of $\a_1$ is at least the entropy of $\a_2$ then there exists a factor map from $\a_1$ to $\a_2$.
\item If $G$ is non-amenable then all of the actions in $\cB$ factor onto each other.
\item The class $\cB$ is closed under direct product.
\end{enumerate}
In the special case that $G$ is amenable and sofic entropy is replaced with spatial entropy, Ornstein-Weiss \cite{OW87} offer a positive answer: the finitely determined processes satisfy these conditions.

\item Gaussian analytic functions (GAFs) are random holomorphic functions on domains in $\C$ \cite{MR2552864}. We may choose the domain to be the unit disk. In that case, there exists a 1-parameter family of Gaussian analytic functions whose law is invariant under group of M\o bius transformations which preserve the disk. This group can be identified with $\textrm{PSL}(2,\R)$. So we have a 1-parameter family of $\textrm{PSL}(2,\R)$-invariant Borel probability measures on the space of holomorphic functions on the unit disk. What is the sofic entropy of these actions?


\item The Markovian triangulation is a random triangulation of the hyperbolic plane with $\textrm{PSL}(2,\R)$-invariant law \cite{MR3055763}. What is its sofic entropy? 


\end{enumerate}

\appendix

\section{Notation index}\label{S:ni}

\begin{itemize}
\item $G$ is an lcsc group with left-Haar measure $\Haar$, left-invariant proper metric $d_G$ and identity $1_G$;
\item $M$ or $M_i$ is a local $G$-space with canonical measure $\vol=\vol_M$;
\item $\Si=\{M_i\}_{i=1}^\infty$ is a sofic approximation to $G$;
\item $U \subset G$ usually represents an open pre-compact neighborhood of the identity;
\item $M[U] \subset M$ is the set of points $p$ which have `nice' $U$-neighborhoods as in Definition \ref{D:sofic group};  
\item $B(r)=B(1_G,r)$ is the open ball of radius $r$ centered at the identity in $G$.
\end{itemize}







\bibliography{biblio}

\def\cprime{$'$} \def\cprime{$'$} \def\cprime{$'$}
  \def\cfudot#1{\ifmmode\setbox7\hbox{$\accent"5E#1$}\else
  \setbox7\hbox{\accent"5E#1}\penalty 10000\relax\fi\raise 1\ht7
  \hbox{\raise.1ex\hbox to 1\wd7{\hss.\hss}}\penalty 10000 \hskip-1\wd7\penalty
  10000\box7} \def\cprime{$'$} \def\cprime{$'$} \def\cprime{$'$}
  \def\cprime{$'$} \def\cprime{$'$} \def\cprime{$'$} \def\cprime{$'$}
\begin{thebibliography}{HKPV09}

\bibitem[AB19]{MR3996050}
Tim Austin and Peter Burton.
\newblock Uniform mixing and completely positive sofic entropy.
\newblock {\em J. Anal. Math.}, 138(2):597--612, 2019.

\bibitem[Avn10]{avni-2010}
Nir Avni.
\newblock Entropy theory for cross-sections.
\newblock {\em Geom. Funct. Anal.}, 19(6):1515--1538, 2010.

\bibitem[BB22]{MR4555895}
Lewis Bowen and Peter Burton.
\newblock Locally compact sofic groups.
\newblock {\em Israel J. Math.}, 251(1):239--270, 2022.

\bibitem[Bow71]{MR0274707}
Rufus Bowen.
\newblock Entropy for group endomorphisms and homogeneous spaces.
\newblock {\em Trans. Amer. Math. Soc.}, 153:401--414, 1971.

\bibitem[Bow10]{bowen-jams-2010}
Lewis Bowen.
\newblock Measure conjugacy invariants for actions of countable sofic groups.
\newblock {\em J. Amer. Math. Soc.}, 23(1):217--245, 2010.

\bibitem[Bow20]{MR4138907}
Lewis Bowen.
\newblock Examples in the entropy theory of countable group actions.
\newblock {\em Ergodic Theory Dynam. Systems}, 40(10):2593--2680, 2020.

\bibitem[CL15]{capraro-lupini}
Valerio Capraro and Martino Lupini.
\newblock {\em Introduction to {S}ofic and hyperlinear groups and {C}onnes'
  embedding conjecture}, volume 2136 of {\em Lecture Notes in Mathematics}.
\newblock Springer, Cham, 2015.
\newblock With an appendix by Vladimir Pestov.

\bibitem[CW13]{MR3055763}
Nicolas Curien and Wendelin Werner.
\newblock The {M}arkovian hyperbolic triangulation.
\newblock {\em J. Eur. Math. Soc. (JEMS)}, 15(4):1309--1341, 2013.

\bibitem[Fel80]{MR597458}
Jacob Feldman.
\newblock {$r$}-entropy, equipartition, and {O}rnstein's isomorphism theorem in
  {${\bf R}^{n}$}.
\newblock {\em Israel J. Math.}, 36(3-4):321--345, 1980.

\bibitem[Hay16]{hayes-fk-determinants}
Ben Hayes.
\newblock Fuglede-{K}adison determinants and sofic entropy.
\newblock {\em Geom. Funct. Anal.}, 26(2):520--606, 2016.

\bibitem[Hay17]{MR3693125}
Ben Hayes.
\newblock Sofic entropy of {G}aussian actions.
\newblock {\em Ergodic Theory Dynam. Systems}, 37(7):2187--2222, 2017.

\bibitem[HKPV09]{MR2552864}
J.~Ben Hough, Manjunath Krishnapur, Yuval Peres, and B\'{a}lint Vir\'{a}g.
\newblock {\em Zeros of {G}aussian analytic functions and determinantal point
  processes}, volume~51 of {\em University Lecture Series}.
\newblock American Mathematical Society, Providence, RI, 2009.

\bibitem[Ker14]{kerr-cpe}
David Kerr.
\newblock Bernoulli actions of sofic groups have completely positive entropy.
\newblock {\em Israel J. Math.}, 202(1):461--474, 2014.

\bibitem[Kie75]{kieffer-1975a}
J.~C. Kieffer.
\newblock A generalized {S}hannon-{M}c{M}illan theorem for the action of an
  amenable group on a probability space.
\newblock {\em Ann. Probability}, 3(6):1031--1037, 1975.

\bibitem[KL11a]{MR2813530}
David Kerr and Hanfeng Li.
\newblock Bernoulli actions and infinite entropy.
\newblock {\em Groups Geom. Dyn.}, 5(3):663--672, 2011.

\bibitem[KL11b]{kerr-li-variational}
David Kerr and Hanfeng Li.
\newblock Entropy and the variational principle for actions of sofic groups.
\newblock {\em Invent. Math.}, 186(3):501--558, 2011.

\bibitem[KL13]{kerr-li-sofic-amenable}
David Kerr and Hanfeng Li.
\newblock Soficity, amenability, and dynamical entropy.
\newblock {\em Amer. J. Math.}, 135(3):721--761, 2013.

\bibitem[KL16]{MR3616077}
David Kerr and Hanfeng Li.
\newblock {\em Ergodic theory}.
\newblock Springer Monographs in Mathematics. Springer, Cham, 2016.
\newblock Independence and dichotomies.

\bibitem[Kol58]{kolmogorov-1958}
A.~N. Kolmogorov.
\newblock A new metric invariant of transient dynamical systems and
  automorphisms in {L}ebesgue spaces.
\newblock {\em Dokl. Akad. Nauk SSSR (N.S.)}, 119:861--864, 1958.

\bibitem[Kol59]{kolmogorov-1959}
A.~N. Kolmogorov.
\newblock Entropy per unit time as a metric invariant of automorphisms.
\newblock {\em Dokl. Akad. Nauk SSSR}, 124:754--755, 1959.

\bibitem[Orn70]{ornstein-1970a}
Donald Ornstein.
\newblock Bernoulli shifts with the same entropy are isomorphic.
\newblock {\em Advances in Math.}, 4:337--352 (1970), 1970.

\bibitem[OW80]{OW80}
Donald~S. Ornstein and Benjamin Weiss.
\newblock Ergodic theory of amenable group actions. {I}. {T}he {R}ohlin lemma.
\newblock {\em Bull. Amer. Math. Soc. (N.S.)}, 2(1):161--164, 1980.

\bibitem[OW87]{OW87}
Donald~S. Ornstein and Benjamin Weiss.
\newblock Entropy and isomorphism theorems for actions of amenable groups.
\newblock {\em J. Analyse Math.}, 48:1--141, 1987.

\bibitem[Pes08]{pestov-sofic-survey}
Vladimir~G. Pestov.
\newblock Hyperlinear and sofic groups: a brief guide.
\newblock {\em Bull. Symbolic Logic}, 14(4):449--480, 2008.

\bibitem[Sin64]{sinai-weak}
Ja.~G. Sina{\u\i}.
\newblock On a weak isomorphism of transformations with invariant measure.
\newblock {\em Mat. Sb. (N.S.)}, 63 (105):23--42, 1964.

\end{thebibliography}
\bibliographystyle{alpha}

\end{document}